\documentclass{article}%
\usepackage{amsmath}
\usepackage{amsfonts}
\usepackage{amssymb}
\usepackage{graphicx}%
\providecommand{\U}[1]{\protect\rule{.1in}{.1in}}
\newtheorem{theorem}{Theorem}

\newtheorem{corollary}[theorem]{Corollary}

\newtheorem{definition}[theorem]{Definition}
\newtheorem{example}[theorem]{Example}

\newtheorem{lemma}[theorem]{Lemma}
\newtheorem{notation}[theorem]{Notation}

\newtheorem{proposition}[theorem]{Proposition}
\newtheorem{remark}[theorem]{Remark}

\newenvironment{proof}[1][Proof]{\noindent\textbf{#1.} }{\ \rule{0.5em}{0.5em}}
\numberwithin{equation}{section}
\numberwithin{theorem}{section}
\begin{document}

\title{Space regularity for evolution operators modeled on H\"{o}rmander vector
fields with time dependent measurable coefficients}
\author{Marco Bramanti}
\maketitle

\begin{abstract}
We consider a heat-type operator $\mathcal{L}$ structured on the left
invariant $1$-homogeneous vector fields which are generators of a Carnot
group, multiplied by a uniformly positive matrix of bounded measurable
coefficients depending only on time. We prove that if $\mathcal{L}u$ is smooth
with respect to the space variables, the same is true for $u$, with
quantitative regularity estimates in the scale of Sobolev spaces defined by
right invariant vector fields. Moreover, the solution and its space
derivatives satisfy a $1/2$-H\"{o}lder continuity estimate with respect to
time. The result is proved both for weak solutions and for distributional
solutions, in a suitable sense\footnote{
H\"{o}rmander's theorem.}.

\end{abstract}

Let $\mathbb{G}=\left(  \mathbb{R}^{N},\circ,D_{\lambda}\right)  $ a Carnot
group and let $X_{1},...,X_{q}$ be the generators of its Lie algebra, so that
the canonical sublaplacian%
\[
\sum_{i=1}^{q}X_{i}^{2}%
\]
and the corresponding heat operator%
\[
\sum_{i=1}^{q}X_{i}^{2}-\partial_{t}%
\]
are hypoelliptic in $\mathbb{R}^{N}$ and $\mathbb{R}^{N+1}$, respectively
(precise definitions will be given in Section \ref{Sec Carnot}). Let us now
consider%
\begin{equation}
\mathcal{L}=\sum_{i,j=1}^{q}a_{ij}\left(  t\right)  X_{i}X_{j}-\partial_{t}
\label{heat}%
\end{equation}
where $\left\{  a_{ij}\left(  t\right)  \right\}  _{i,j=1}^{q}$ is a real
symmetric matrix of bounded measurable coefficients, uniformly positive:%
\begin{equation}
\nu\left\vert \xi\right\vert ^{2}\leqslant\sum_{i,j=1}^{q}a_{ij}\left(
t\right)  \xi_{i}\xi_{j}\leqslant\nu^{-1}\left\vert \xi\right\vert ^{2}
\label{ellipticity}%
\end{equation}
for every $\xi\in\mathbb{R}^{q}$, a.e. $t\in\left(  0,T\right)  $. We want to
prove a regularity result for $\mathcal{L}$ in the space variables, that is,
roughly speaking: if $u\in W^{1,2}\left(  \left(  0,T\right)  ,L_{loc}%
^{2}\left(  \mathbb{R}^{N}\right)  \right)  $ is a weak solution to
$\mathcal{L}u=F$, $u\left(  0,\cdot\right)  =0$ and $F$ is smooth, \emph{with
respect to the space variables}, in some domain $\left(  0,T\right)
\times\Omega$, then the same is true for $u$, with quantitative regularity
estimates on $u$ in terms of $\mathcal{L}u$. Also, we will prove that, if $F$
is smooth w.r.t. the space variables, then $u$ and every space derivative
$\partial_{x}^{\alpha}u$ are $\frac{1}{2}$-H\"{o}lder continuous with respect
to $t$. See Theorems \ref{Thm subsubelliptic} and \ref{Thm subell plus} for
the precise statements. This kind of regularity is the best we can hope, even
for a uniformly parabolic operator%
\[
\mathcal{L}u=u_{t}-a\left(  t\right)  u_{xx}%
\]
as soon as $a$ is only $L^{\infty}$ (see Example \ref{contres parabolico}).
The above regularity result can be extended also to distributional solutions
belonging to $W^{1,2}\left(  \left(  0,T\right)  ,\mathcal{D}^{\prime}\left(
\mathbb{R}^{N}\right)  \right)  $ (see Theorem \ref{Thm hypo} for the precise
statement). This can be seen as a kind of H\"{o}rmander's theorem with respect
to the space variables.

Results of this kind have been proved by Krylov \cite{K1}, who considered
operators%
\[
\mathcal{L}=\partial_{t}-\sum_{k=1}^{q}L_{k}^{2}+L_{0}%
\]
with%
\[
L_{k}=\sum_{i=1}^{N}\sigma^{ik}\left(  t,x\right)  \partial_{x_{i}}%
\]
where the functions $\sigma^{ik}\left(  t,x\right)  $ are assumed to have
$x$-derivatives of every order uniformly bounded for $x\in\mathbb{R}^{N}$ and
$t\in\left(  0,1\right)  $, and the vector fields $L_{0},L_{1},...,L_{q}$ for
every fixed $t$ satisfy H\"{o}rmander's condition in $\mathbb{R}^{N}$. Now,
every operator (\ref{heat}) can be rewritten as%
\[
-\mathcal{L}=\partial_{t}-\sum_{k=1}^{q}L_{k}^{2}%
\]
with%
\[
\sigma^{ik}\left(  t,x\right)  =\sum_{j=1}^{q}m_{jk}\left(  t\right)
b_{ji}\left(  x\right)
\]
where
\[
X_{j}=\sum_{i=1}^{N}b_{ji}\left(  x\right)  \partial_{x_{i}}%
\]
and%
\[
a_{ij}\left(  t\right)  =\sum_{k=1}^{q}m_{ik}\left(  t\right)  m_{jk}\left(
t\right)
\]
so that%
\begin{align*}
D_{x}^{\alpha}\sigma^{ik}\left(  t,x\right)   &  =\sum_{j=1}^{q}m_{jk}\left(
t\right)  D_{x}^{\alpha}b_{ji}\left(  x\right) \\
\left\vert D_{x}^{\alpha}\sigma^{ik}\left(  t,x\right)  \right\vert  &
\leqslant c_{\nu}\sum_{j=1}^{q}\left\vert D_{x}^{\alpha}b_{ji}\left(
x\right)  \right\vert .
\end{align*}
Since the coefficients $b_{ji}\left(  x\right)  $ of the generators on a
Carnot group are polynomials, the functions $\left\vert D_{x}^{\alpha}%
b_{ji}\left(  x\right)  \right\vert $ are not \emph{globally} bounded on
$\mathbb{R}^{N}$. Therefore, although the class of operators that we consider
is strictly contained in the class considered by Krylov as to their structure,
the assumption on $\sigma^{ik}\left(  t,x\right)  $ made in \cite{K1} is not
satisfied in our situation.

Actually, the technique employed in this paper is very different from that in
\cite{K1}. In \cite{K1}, following the classical approach introduced by Kohn
\cite{Koh} and Ole\u{\i}nik-Radkevi\v{c} \cite{OlRa}, pseudodifferential
operators and Sobolev spaces of fractional order are used. Here, instead, we
adapt to the evolutionary case the technique introduced in \cite{BB1} to give
a proof of H\"{o}rmander's theorem for sublaplacians on Carnot groups. The
main idea consists in measuring the regularity of solutions of an equation
$\mathcal{L}u=f$, where $\mathcal{L}$ is a \emph{left invariant} operator, in
terms of Sobolev spaces induced by \emph{right invariant }vector fields. Since
a right invariant and a left invariant operator always commute, this approach
greatly simplifies the proof of higher order estimates. We handle Sobolev
norms with respect to vector fields by means of equivalent norms defined in
terms of finite difference operators, in the directions of the vector fields
$X_{1},...,X_{q}$. This feature of our argument is reminiscent of the original
proof of H\"{o}rmander's theorem given in \cite{Hor2}, although in the richer
framework of Carnot groups the proof becomes much simpler.

Let us now give some motivation for the present research and describe some
related literature. The regularity result proved in \cite{K1} has been applied
by the same Author in \cite{K3} to prove an analogous result for stochastic
PDEs, and in \cite{K4}, in the context of filtering problems. We refer to
\cite{K3} for motivations to prove this result without any continuity
assumption on the coefficients with respect to time.

Hyperbolic operators of the kind
\[
Hu=u_{tt}-\sum_{i,j=1}^{n}a_{ij}\left(  t\right)  u_{x_{i}x_{j}}%
\]
with merely bounded measurable $a_{ij}$ have been studied by many authors, see
for instance \cite{CDGS}, \cite{CDSR}, \cite{GR} and references therein. In
particular, \cite{GR} gives some physical motivation to study this class of
operators under no regularity condition on $a_{ij}\left(  t\right)  $.

Operators of the kind%
\begin{equation}
\mathcal{L}=\sum_{i,j=1}^{q}a_{ij}\left(  t,x\right)  X_{i}X_{j}-\partial_{t},
\label{nondivergence heat}%
\end{equation}
satisfying (\ref{ellipticity}) have been studied by several Authors, assuming
the coefficients $a_{ij}\left(  t,x\right)  $ either H\"{o}lder continuous or
with vanishing mean oscillation, and proving a priori estimates and regularity
results in the scale of H\"{o}lder or Sobolev spaces induced by the vector
fields $\left\{  X_{i}\right\}  _{i=1}^{q}$ and the distance they induce. See
for instance \cite{BB3}, \cite{BBLU}, \cite{BZ} and references therein. In
\cite{BBLU}, for the operator $\mathcal{L}$ with H\"{o}lder continuous
coefficients, a heat kernel has been constructed and shown to satisfy sharp
Gaussian estimates, which also imply a scale invariant Harnack inequality.

The operators (\ref{heat}) studied in the present paper can also be seen as
model operators to study the more general class (\ref{nondivergence heat})
with the coefficients satisfying some moderate regularity assumtpion in $x$,
but only $L^{\infty}$ with respect to time, an area of research that we plan
to attack in the future.

\section{Preliminaries about Carnot groups\label{Sec Carnot}}

Let us recall some standard definitions and results that will be useful in the
following. For the proofs of these facts the reader is referred to \cite{Fo2},
\cite[Chap.1]{BLUb}. A \emph{homogeneous group} (in $\mathbb{R}^{N}$) is a Lie
group $\left(  \mathbb{R}^{N},\circ\right)  $ (where the group operation
$\circ$ will be thought as a \textquotedblleft translation\textquotedblright)
endowed with a one parameter family $\left\{  D_{\lambda}\right\}
_{\lambda>0}$ of group automorphisms (\textquotedblleft
dilations\textquotedblright) which act this way:%
\begin{equation}
D_{\lambda}\left(  x_{1},x_{2},...,x_{N}\right)  =\left(  \lambda^{\alpha_{1}%
}x_{1},\lambda^{\alpha_{2}}x_{2},...,\lambda^{\alpha_{N}}x_{N}\right)
\label{dilations1}%
\end{equation}
for suitable integers $1=\alpha_{1}\leqslant\alpha_{2}\leqslant...\leqslant
\alpha_{N}$. We will write $\mathbb{G}=\left(  \mathbb{R}^{N},\circ
,D_{\lambda}\right)  $ to denote this structure. The number%
\[
Q=\sum_{i=1}^{N}\alpha_{i}%
\]
will be called \emph{homogeneous dimension }of $\mathbb{G}$. A
\emph{homogeneous norm} on $\mathbb{G}$ is a continuous function%
\[
\left\Vert \cdot\right\Vert :\mathbb{G}\rightarrow\lbrack0,+\infty),
\]
such that, for some constant $c>0$ and every $x,y\in\mathbb{G},$%
\[%
\begin{tabular}
[c]{ll}%
$\left(  i\right)  $ & $\left\Vert x\right\Vert =0\Longleftrightarrow x=0$\\
$\left(  ii\right)  $ & $\left\Vert D_{\lambda}\left(  x\right)  \right\Vert
=\lambda\left\Vert x\right\Vert \text{ }\forall\lambda>0$\\
$\left(  iii\right)  $ & $\left\Vert x^{-1}\right\Vert \leqslant c\left\Vert
x\right\Vert $\\
$\left(  iv\right)  $ & $\left\Vert x\circ y\right\Vert \leqslant c\left(
\left\Vert x\right\Vert +\left\Vert y\right\Vert \right)  .$%
\end{tabular}
\ \ \ \ \ \ \
\]
We will always use the symbol $\left\Vert \cdot\right\Vert $, without any
subscript, to denote a homogeneous norm in $\mathbb{G}$. Examples of
homogeneous norms are the following:%
\[
\left\Vert x\right\Vert =\max_{k=1,2,...,N}\left\vert x_{k}\right\vert
^{\frac{1}{\alpha_{k}}}%
\]
or%
\[
\left\Vert x\right\Vert =\left(  \sum_{k=1}^{N}\left\vert x_{k}\right\vert
^{\frac{Q}{\alpha_{k}}}\right)  ^{1/Q}.
\]
It can be proved that any two homogeneous norms on $\mathbb{G}$ are equivalent.

We say that a smooth function $f$ in $\mathbb{G}\setminus\left\{  0\right\}  $
is $D_{\lambda}$-\emph{homogeneous of degree} $\beta\in\mathbb{R}$ (or simply
\textquotedblleft$\beta$-homogeneous\textquotedblright) if%
\[
f\left(  D_{\lambda}\left(  x\right)  \right)  =\lambda^{\beta}f\left(
x\right)  \text{ \ \ }\forall\lambda>0,x\in\mathbb{G}\setminus\left\{
0\right\}  \text{.}%
\]

Given any differential operator $P$ with smooth coefficients on $\mathbb{G}$,
we say that $P$ is \emph{left invariant} if for every $x,y\in\mathbb{G}$ and
every smooth function $f$%
\[
P\left(  L_{y}f\right)  \left(  x\right)  =L_{y}\left(  Pf\left(  x\right)
\right)  ,
\]
where%
\[
L_{y}f\left(  x\right)  =f\left(  y\circ x\right)  .
\]
Analogously one defines the notion of \emph{right invariant} differential
operator. Also, $P$ is said $\beta$-\emph{homogeneous} (for some $\beta
\in\mathbb{R}$) if%
\[
P\left(  f\left(  D_{\lambda}\left(  x\right)  \right)  \right)
=\lambda^{\beta}\left(  Pf\right)  \left(  D_{\lambda}\left(  x\right)
\right)
\]
for every smooth function $f$, $\lambda>0$ and $x\in\mathbb{G}$.

A \emph{vector field }is a first order differential operator
\[
X=\sum_{i=1}^{N}c_{i}\left(  x\right)  \partial_{x_{i}}.
\]

Let $\mathfrak{g}$ be the \emph{Lie algebra of left invariant vector fields}
over $\mathbb{G}$, where the Lie bracket of two vector fields is defined as
usual by%
\[
\left[  X,Y\right]  =XY-YX.
\]

Let us denote by $X_{1},X_{2},\ldots,X_{N}$ the \emph{canonical base} of
$\mathfrak{g}$, that is for $i=1,2,....,N$, $X_{i}$ is the only left invariant
vector field that agrees with $\partial_{x_{i}}$ at the origin. Also,
$X_{1}^{R},X_{2}^{R},\ldots,X_{N}^{R}$ will denote the right invariant vectors
fields that agree with $\partial_{x_{1}},\partial_{x_{2}},...,\partial_{x_{N}%
}$ (and hence with $X_{1},X_{2},\ldots,X_{N}$) at the origin.

We assume that for some integer $q<N$ the vector fields $X_{1},X_{2}%
,\ldots,X_{q}$ are $1$-homogeneous and the Lie algebra generated by them is
$\mathfrak{g}$. If $s$ is the maximum length of commutators
\[
\left[  X_{i_{1}},\left[  X_{i_{2}},...,\left[  X_{i_{s-1}},X_{i_{s}}\right]
\right]  \right]  ,\text{ \ }i_{j}\in\left\{  1,2,...,q\right\}
\]
required to span $\mathfrak{g}$, then we will say that $\mathfrak{g}$ is a
\emph{stratified Lie algebra of step} $s$, $\mathbb{G}$ is a \emph{Carnot
group} (or a stratified homogeneous group) and its \emph{generators }%
$X_{1},X_{2},\ldots,X_{q}$ satisfy H\"{o}rmander's condition at step $s$ in
$\mathbb{G}$. Under these assumptions, by H\"{o}rmander's theorem (see
\cite{Hor2}), the\emph{ canonical sublaplacian}%
\[
L=\sum_{i=1}^{q}X_{i}^{2}%
\]
is hypoelliptic in $\mathbb{R}^{N}$, that is: for every domains $\Omega
^{\prime}\subset\Omega\subset\mathbb{R}^{N}$, whenever $u\in\mathcal{D}%
^{\prime}\left(  \Omega\right)  $ solves in distributional sense the equation
$Lu=f$ in $\Omega$, then $f\in C^{\infty}\left(  \Omega^{\prime}\right)
\Rightarrow u\in C^{\infty}\left(  \Omega^{\prime}\right)  .$

Analogously, the corresponding heat operator
\[
H=\sum_{i=1}^{q}X_{i}^{2}-\partial_{t}%
\]
is hypoelliptic in $\mathbb{R}^{N+1}$.

We will make use of the \emph{Sobolev spaces} $W_{X}^{k,p}\left(
\mathbb{G}\right)  $, $W_{X^{R}}^{k,p}\left(  \mathbb{G}\right)  $
\emph{induced by the systems of vector fields}%
\[
X=\left\{  X_{1},X_{2},\ldots,X_{q}\right\}  ,X^{R}=\left\{  X_{1}^{R}%
,X_{2}^{R},\ldots,X_{q}^{R}\right\}  ,
\]
respectively. More precisely, given an open subset $\Omega$ of $\mathbb{R}%
^{N}$, we say that $f\in W_{X}^{1,2}\left(  \Omega\right)  $ if $f\in
L^{2}\left(  \Omega\right)  $ and there exist, in weak sense, $X_{j}f\in
L^{2}\left(  \Omega\right)  $ for $j=1,2,...,q.$ Inductively, we say that
$f\in W_{X}^{k,2}\left(  \Omega\right)  $ for $k=2,3,...$ if $f\in
W_{X}^{k-1,2}\left(  \Omega\right)  $ and any weak derivative of order $k-1$
of $f$, $X_{j_{1}}X_{j_{2}}...X_{j_{k-1}}f$, belongs to $W_{X}^{1,2}\left(
\Omega\right)  $. We set%
\[
\left\Vert f\right\Vert _{W_{X}^{k,2}\left(  \Omega\right)  }=\left\Vert
f\right\Vert _{L^{2}\left(  \Omega\right)  }+\sum_{h=1}^{k}\sum_{j_{i}%
=1,2,...,q}\left\Vert X_{j_{1}}X_{j_{2}}...X_{j_{h}}f\right\Vert
_{L^{2}\left(  \Omega\right)  }.
\]

The space $W_{X^{R}}^{k,2}\left(  \Omega\right)  $ has a similar definition.
We will also use local Sobolev spaces. For example, we will say that $f\in
W_{X,loc}^{k,2}\left(  \Omega\right)  $ if for every $\varphi\in C_{0}%
^{\infty}\left(  \Omega\right)  $, we have $\varphi f\in W_{X}^{k,2}\left(
\Omega\right)  $.

For homogeneity reasons, the generators $X_{1},...,X_{q}$ satisfy the simple
relation $X_{i}^{\ast}=-X_{i}$ (where $X^{\ast}$ stands for the transposed
operator of $X$). In other words,%
\begin{equation}
\int_{\mathbb{G}}f\left(  X_{i}g\right)  =-\int_{\mathbb{G}}\left(
X_{i}f\right)  g \label{transposed}%
\end{equation}
whenever $f\in W_{X,loc}^{1,2}\left(  \mathbb{G}\right)  $ and $g\in C_{0}%
^{1}\left(  \mathbb{G}\right)  $.

The validity of H\"{o}rmander's condition at step $s$ implies the following important:

\begin{proposition}
[{See \cite[Prop. 2.1]{BB1}}]\label{Prop reg sobolev}Under the above
assumptions we have:

1.
\[%
{\textstyle\bigcap\limits_{k=1}^{\infty}}
W_{X}^{k,2}\left(  \Omega\right)  \subset C^{\infty}\left(  \Omega\right)  .
\]

2. For any positive integer $k$ and any $\Omega^{\prime}\Subset\Omega$ there
exists a constant $c>0$ such that, for every $u\in W_{X}^{ks,2}\left(
\Omega\right)  $ we have%
\[
\left\Vert u\right\Vert _{W^{k,2}\left(  \Omega^{\prime}\right)  }\leqslant
c\left\Vert u\right\Vert _{W_{X}^{ks,2}\left(  \Omega\right)  },
\]
where $W^{k,2}\left(  \Omega^{\prime}\right)  $ denotes the standard Sobolev space.
\end{proposition}

Let us point out a relation between left and right invariant operators which
will be very useful in the following.

\begin{proposition}
[{see \cite[Prop. 2.2]{BB1}}]\label{Prop left and right commute}Let
$\mathcal{L},\mathcal{R}$ be any two differential operators on $\mathbb{G}$
with smooth coefficients, left and right invariant, respectively. Then
$\mathcal{L}$ and $\mathcal{R}$ commute:%
\[
\mathcal{LR}f=\mathcal{RL}f
\]
for any smooth function $f.$
\end{proposition}

For every given couple of measurable functions $\varphi,\psi:\mathbb{G}%
\rightarrow\mathbb{R}$ we define%
\[
\varphi\ast\psi\left(  x\right)  =\int_{\mathbb{G}}\varphi\left(  y\right)
\psi\left(  y^{-1}\circ x\right)  dy
\]
whenever the integral makes sense. One can prove the following:

\begin{proposition}
\label{Prop conv distrib}For every couple of measurable functions $f,\psi$
defined on $\mathbb{G}$ such that the following convolutions are well defined,
we have

i) if $\mathcal{P}$ is a left invariant differential operator then%
\begin{equation}
\mathcal{P}\left(  f\ast\psi\right)  =f\ast\mathcal{P\psi}, \label{f*Pg}%
\end{equation}

ii) if $\mathcal{P}$ is a right invariant differential operator then%
\[
\mathcal{P}\left(  \psi\ast f\right)  =\mathcal{P\psi}\ast f
\]
whenever $\mathcal{P\psi}$ exists at least in weak sense.
\end{proposition}

\section{Subelliptic estimates for heat-type operators with $t$-measurable
coefficients}

For a domain $\Omega\subseteq\mathbb{G}$, let
\[
\Omega_{T}=\left(  0,T\right)  \times\Omega.
\]
We are going to define several function spaces on $\mathbb{G}_{T}=\left(
0,T\right)  \times\mathbb{G}$ that we will use in the following.

The definitions of the spaces $L^{2}\left(  \left(  0,T\right)  ,X\right)  $,
$W^{1,2}\left(  \left(  0,T\right)  ,X\right)  $, $C^{0}\left(  \left[
0,T\right]  ,X\right)  $ when $X$ is a Banach space are standard. For
instance, we will often use the spaces%
\[
L^{2}\left(  \left(  0,T\right)  ,W_{X}^{k,2}\left(  \mathbb{G}\right)
\right)
\]
(for $k=1,2,3,...$) normed with%
\[
\left\Vert f\right\Vert _{L^{2}\left(  \left(  0,T\right)  ,W_{X}^{k,2}\left(
\mathbb{G}\right)  \right)  }=\left\Vert f\right\Vert _{L^{2}\left(
\mathbb{G}_{T}\right)  }+\sum_{j=1}^{k}\sum_{i_{1},...,i_{j}\in\left\{
1,...,q\right\}  }\left\Vert X_{i_{1}}X_{i_{2}}...X_{i_{j}}f\right\Vert
_{L^{2}\left(  \mathbb{G}_{T}\right)  }%
\]
and the analogous spaces $L^{2}\left(  \left(  0,T\right)  ,W_{X^{R}}%
^{k,2}\left(  \mathbb{G}\right)  \right)  $. 

We will say that $u\in L^{2}\left(  \left(  0,T\right)  ,W_{X,loc}%
^{k,2}\left(  \mathbb{G}\right)  \right)  $ when for every $\zeta\in
C_{0}^{\infty}\left(  \mathbb{G}\right)  $ we have $u\zeta\in L^{2}\left(
\left(  0,T\right)  ,W_{X}^{k,2}\left(  \mathbb{G}\right)  \right)  $.

For a function $f\in L^{2}\left(  \left(  0,T\right)  ,W_{X}^{1,2}\left(
\mathbb{G}\right)  \right)  $ we will also use the shorthand notation%
\[
\left\Vert \nabla_{X}f\right\Vert _{L^{2}\left(  \mathbb{G}_{T}\right)  }%
^{2}=\sum_{i=1}^{q}\left\Vert X_{i}f\right\Vert _{L^{2}\left(  \mathbb{G}%
_{T}\right)  }^{2},
\]
with the analogous meaning for $\left\Vert \nabla_{X^{R}}f\right\Vert
_{L^{2}\left(  \mathbb{G}_{T}\right)  }^{2}.$

\begin{definition}
We say that a function $u$ belongs to $L^{2}\left(  \left(  0,T\right)
,C^{\infty}\left(  \overline{\Omega}\right)  \right)  $ if $u\in L^{2}\left(
\left(  0,T\right)  ,C^{k}\left(  \overline{\Omega}\right)  \right)  $ for
every $k=0,1,2,...$ Explicitly, this implies that%
\[
\int_{0}^{T}\left\Vert u\left(  t,\cdot\right)  \right\Vert _{C^{k}\left(
\overline{\Omega}\right)  }^{2}dt<\infty\text{ for every }k=0,1,2,...
\]

We say that a function $u$ belongs to $C^{0}\left(  \left[  0,T\right]
,C^{\infty}\left(  \overline{\Omega}\right)  \right)  $ if $u\in C^{0}\left(
\left[  0,T\right]  ,C^{k}\left(  \overline{\Omega}\right)  \right)  \ $for
every $k=0,1,2,...$
\end{definition}

\begin{definition}
We let:%
\begin{align*}
\mathcal{H} &  =L^{2}\left(  \left(  0,T\right)  ,W_{X}^{2,2}\left(
\mathbb{G}\right)  \right)  \cap W^{1,2}\left(  \left(  0,T\right)
,L^{2}\left(  \mathbb{G}\right)  \right)  \\
&  =\left\{  u\in L^{2}\left(  \mathbb{G}_{T}\right)  :u_{t},X_{i}u,X_{i}%
X_{j}u\in L^{2}\left(  \mathbb{G}_{T}\right)  \right\}  .
\end{align*}
Note that $\mathcal{H}\subset C^{0}\left(  \left[  0,T\right]  ,L^{2}\left(
\mathbb{G}\right)  \right)  $, so that for $u\in\mathcal{H}$ and $t\in\left[
0,T\right]  $, $u\left(  t,\cdot\right)  $ is a well defined element of
$L^{2}\left(  \mathbb{G}\right)  $.

We will also use%
\[
\mathcal{H}_{0}=\left\{  u\in W^{1,2}\left(  \left(  0,T\right)  ,L_{loc}%
^{2}\left(  \mathbb{G}\right)  \right)  :\forall\phi\in C_{0}^{\infty}\left(
\mathbb{G}\right)  \text{ }u\phi\in\mathcal{H}\text{ and }\left(
u\phi\right)  \left(  0,\cdot\right)  =0\right\}  .
\]

\end{definition}

\begin{proposition}
Let $\mathcal{L}$ be as in (\ref{heat}) and let (\ref{ellipticity}) be in
force. Then for every $u\in\mathcal{H}$ such that $u\left(  0,\cdot\right)
=0$ we have%
\begin{equation}
\left\Vert \nabla_{X}u\right\Vert _{L^{2}\left(  \mathbb{G}_{T}\right)
}\leqslant c_{\nu}\left\{  \left\Vert \mathcal{L}u\right\Vert _{L^{2}\left(
\mathbb{G}_{T}\right)  }+\left\Vert u\right\Vert _{L^{2}\left(  \mathbb{G}%
_{T}\right)  }\right\}  \label{basic Sobolev}%
\end{equation}
for a constant $c_{\nu}$ only depending on the ellipticity constant $\nu$ in
(\ref{ellipticity}).
\end{proposition}

\begin{proof}
For $u\in\mathcal{H}$ we have, recalling that $X_{i}^{\ast}=-X_{i}$ (see
(\ref{transposed})):%
\begin{align}
&  -\int\int_{\mathbb{G}_{T}}\left(  u\mathcal{L}u\right)  dtdx=\int%
\int_{\mathbb{G}_{T}}\left(  u\partial_{t}u\right)  dtdx-\int\int%
_{\mathbb{G}_{T}}\left(  u\sum_{i,j=1}^{q}a_{ij}\left(  t\right)  X_{i}%
X_{j}u\right)  dtdx\nonumber\\
&  =\frac{1}{2}\int_{\mathbb{G}}\left(  \int_{0}^{T}\partial_{t}\left(
u^{2}\right)  dt\right)  dx-\sum_{i,j=1}^{q}\int_{0}^{T}a_{ij}\left(
t\right)  \left(  \int_{\mathbb{G}}\left(  uX_{i}X_{j}u\right)  dx\right)
dt\nonumber\\
&  =\frac{1}{2}\int_{\mathbb{G}}\left(  u^{2}\left(  T,x\right)  -u^{2}\left(
0,x\right)  \right)  dx+\sum_{i,j=1}^{q}\int_{0}^{T}a_{ij}\left(  t\right)
\left(  \int_{\mathbb{G}}\left(  X_{i}uX_{j}u\right)  dx\right)
dt.\label{first identity}%
\end{align}
Since
\[
\sum_{i,j=1}^{q}\int_{0}^{T}a_{ij}\left(  t\right)  \left(  \int_{\mathbb{G}%
}\left(  X_{i}uX_{j}u\right)  dx\right)  dt\geq\nu\sum_{i=1}^{q}\int_{0}%
^{T}\int_{\mathbb{G}}\left(  X_{i}u\right)  ^{2}dxdt
\]
we have%
\begin{equation}
\left\Vert \nabla_{X}u\right\Vert _{L^{2}\left(  \mathbb{G}_{T}\right)  }%
^{2}\leqslant\frac{1}{\nu}\left\Vert \mathcal{L}u\right\Vert _{L^{2}\left(
\mathbb{G}_{T}\right)  }\left\Vert u\right\Vert _{L^{2}\left(  \mathbb{G}%
_{T}\right)  }+\frac{1}{2\nu}\left\Vert u\left(  0,\cdot\right)  \right\Vert
_{L^{2}\left(  \mathbb{G}\right)  }.\label{basic}%
\end{equation}
In particular, for $u$ vanishing on $t=0$ we get (\ref{basic Sobolev}).
\end{proof}

In the following of this section we will recall and adapt several definitions
and arguments taken from \cite{BB1}. The reader is referred to that paper for
some details.

\begin{definition}
[Finite difference operators]For every $h\in\mathbb{G}$ and function $f$
defined in $\mathbb{G}$, let us define the operators:%
\begin{align*}
\Delta_{h}f\left(  x\right)   &  =f\left(  x\circ h\right)  -f\left(  x\right)
\\
\widetilde{\Delta}_{h}f\left(  x\right)   &  =f\left(  h\circ x\right)
-f\left(  x\right)  .
\end{align*}
Whenever the function $f$ also depends on $t$, we will simply write%
\[
\Delta_{h}f\left(  t,x\right)  =\Delta_{h}\left[  f\left(  t,\cdot\right)
\right]  \left(  x\right)
\]
and analogously for $\widetilde{\Delta}_{h}f\left(  t,x\right)  .$
\end{definition}

\begin{definition}
For $m=1,2,3,4,...$, let%
\begin{align*}
\Delta_{h}^{m}  &  =\underset{m\text{ times}}{\underbrace{\Delta_{h}\Delta
_{h}...\Delta_{h}}}.\\
\widetilde{\Delta}_{h}^{m}  &  =\underset{m\text{ times}%
}{\underbrace{\widetilde{\Delta}_{h}\widetilde{\Delta}_{h}...\widetilde{\Delta
}_{h}}}.
\end{align*}

Then, for $\alpha>0$ and $f\in L^{2}\left(  \mathbb{G}_{T}\right)  $ we define
the semi-norms%
\begin{align*}
\left\vert f\right\vert _{m,\alpha} &  =\sup\left\{  \frac{\left\Vert
\Delta_{h}^{m}f\right\Vert _{L^{2}\left(  \mathbb{G}_{T}\right)  }}{\left\Vert
h\right\Vert ^{\alpha}}:h=\operatorname{Exp}\left(  tX_{i}\right)  \text{
}\forall i=1,...,q,t\in\mathbb{R}:0<\left\Vert h\right\Vert \leqslant
1\right\}  \\
\left\vert f\right\vert _{m,\alpha}^{R} &  =\sup\left\{  \frac{\left\Vert
\widetilde{\Delta}_{h}^{m}f\right\Vert _{L^{2}\left(  \mathbb{G}_{T}\right)
}}{\left\Vert h\right\Vert ^{\alpha}}:h=\operatorname{Exp}\left(
tX_{i}\right)  \text{ }\forall i=1,...,q,t\in\mathbb{R}:0<\left\Vert
h\right\Vert \leqslant1\right\}  .
\end{align*}
We also set for convenience%
\begin{align*}
\left\vert f\right\vert _{0} &  =\left\vert f\right\vert _{0}^{R}=\left\Vert
f\right\Vert _{L^{2}\left(  \mathbb{G}_{T}\right)  }\\
\left\vert f\right\vert _{m} &  =\left\vert f\right\vert _{m,m}\\
\left\vert f\right\vert _{m}^{R} &  =\left\vert f\right\vert _{m,m}^{R}.
\end{align*}

\end{definition}

The relations between the above seminorms and Sobolev norms with respect to
vector fields are contained in the following two results, which can be derived
by \cite[Thm. 3.11, Prop.3.13]{BB1} simply integrating in $t$.

\begin{proposition}
\label{Thm caratterizzaz increm finiti}For $m=1,2,...$there exists $c=c\left(
m,\mathbb{G}\right)  $ such that, for every $f\in L^{2}\left(  \mathbb{G}%
_{T}\right)  $ we have:

1. If $f\in L^{2}\left(  \left(  0,T\right)  ,W_{X}^{m,2}\left(
\mathbb{G}\right)  \right)  $ then%
\begin{equation}
\sum_{k=0}^{m}\left\vert f\right\vert _{k}\leqslant c\left\Vert f\right\Vert
_{L^{2}\left(  \left(  0,T\right)  ,W_{X}^{m,2}\left(  \mathbb{G}\right)
\right)  }\label{1}%
\end{equation}
Analogously,

2. If $f\in L^{2}\left(  \left(  0,T\right)  ,W_{X^{R}}^{m,2}\left(
\mathbb{G}\right)  \right)  $ then%
\begin{equation}
\sum_{k=0}^{m}\left\vert f\right\vert _{k}^{R}\leqslant c\left\Vert
f\right\Vert _{L^{2}\left(  \left(  0,T\right)  ,W_{X^{R}}^{m,2}\left(
\mathbb{G}\right)  \right)  }.\label{2}%
\end{equation}

\end{proposition}

\begin{proposition}
\label{Lemma 2 caratt incr fin}There exists $C=C\left(  \mathbb{G}\right)  $
such that for every $f\in L^{2}\left(  \mathbb{G}_{T}\right)  $ we
have:\newline1. If $\left\vert f\right\vert _{1}<\infty$ then $f\in
L^{2}\left(  \left(  0,T\right)  ,W_{X}^{1,2}\left(  \mathbb{G}\right)
\right)  $, with%
\[
\left\Vert \nabla_{X}f\right\Vert _{L^{2}\left(  \mathbb{G}_{T}\right)
}\leqslant C\left\vert f\right\vert _{1}.
\]
2. If $\left\vert f\right\vert _{1}^{R}<\infty$ then $f\in L^{2}\left(
\left(  0,T\right)  ,W_{X^{R}}^{1,2}\left(  \mathbb{G}\right)  \right)  $,
with%
\[
\left\Vert \nabla_{X^{R}}f\right\Vert _{L^{2}\left(  \mathbb{G}_{T}\right)
}\leqslant C\left\vert f\right\vert _{1}^{R}.
\]

\end{proposition}

The following bound instead links the $L^{2}\left(  \left(  0,T\right)
,W_{X}^{1,2}\left(  \mathbb{G}\right)  \right)  $ norm with the operators
$\widetilde{\Delta}_{h}$:

\begin{proposition}
Let $\Omega$ be a bounded domain in $\mathbb{G}$. There exists $c=c\left(
\Omega,\mathbb{G}\right)  $ such that for every $u\in L^{2}\left(  \left(
0,T\right)  ,W_{X}^{1,2}\left(  \mathbb{G}\right)  \right)  $ with
$\operatorname{sprt}u\left(  t,\cdot\right)  \subset\Omega$ for every
$t\in\left(  0,T\right)  $ we have%
\[
\left\Vert \widetilde{\Delta}_{h}u\right\Vert _{L^{2}\left(  \mathbb{G}%
_{T}\right)  }\leqslant c\left\Vert h\right\Vert ^{1/s}\left\Vert \nabla
_{X}u\right\Vert _{L^{2}\left(  \mathbb{G}_{T}\right)  }.
\]
(Recall that $s$ is the step of the Lie algebra).
\end{proposition}

\begin{proof}
It is enough to apply to $u\left(  t,\cdot\right)  $ the computations made in
\cite[Prop.3.7, Lemma 3.8]{BB1} for functions in $W_{X}^{1,2}\left(
\mathbb{G}\right)  $ and then integrate on $\left(  0,T\right)  $.
\end{proof}

If $u\in$ $\mathcal{H}$, $u\left(  t,\cdot\right)  $ is supported in some
bounded domain $\Omega$ for every $t\in\left[  0,T\right]  $ and $u\left(
0,\cdot\right)  =0$, then by the previous Proposition and (\ref{basic Sobolev}%
) we get%
\[
\left\Vert \widetilde{\Delta}_{h}u\right\Vert _{L^{2}\left(  \mathbb{G}%
_{T}\right)  }\leqslant c_{\nu}\left\Vert h\right\Vert ^{1/s}\left\{
\left\Vert \mathcal{L}u\right\Vert _{L^{2}\left(  \mathbb{G}_{T}\right)
}+\left\Vert u\right\Vert _{L^{2}\left(  \mathbb{G}_{T}\right)  }\right\}
\]
that is%
\begin{equation}
\left\vert u\right\vert _{1,1/s}^{R}\leqslant c_{\nu}\left\{  \left\Vert
\mathcal{L}u\right\Vert _{L^{2}\left(  \mathbb{G}_{T}\right)  }+\left\Vert
u\right\Vert _{L^{2}\left(  \mathbb{G}_{T}\right)  }\right\}
.\label{first 1/s bound}%
\end{equation}

\begin{notation}
Henceforth, we will write
\[
\zeta_{0}\prec\zeta
\]
if $\zeta_{0},\zeta\in C_{0}^{\infty}\left(  \mathbb{G}\right)  $ such that
$0\leqslant\zeta_{0}\leqslant\zeta\leqslant1$ and $\zeta=1$ on
$\operatorname*{sprt}\zeta_{0}.$
\end{notation}

We have the following analog of Theorem 3.15 in \cite{BB1}:

\begin{theorem}
\label{Thm regularity localized}Let $\zeta_{0},\zeta\in C_{0}^{\infty}\left(
\mathbb{G}\right)  $ with $\zeta_{0}\prec\zeta$. For every $m\in\mathbb{N}$
the exists $c=c\left(  \zeta_{0},\zeta,m,\mathbb{G},\nu\right)  >0$ such that
if $u\in\mathcal{H}_{0}\ $then%
\begin{equation}
\left\vert \zeta_{0}u\right\vert _{m,m/s}^{R}\leqslant c\left(  \sum
_{j=0}^{m-1}\left\vert \zeta\mathcal{L}u\right\vert _{j}^{R}+\left\Vert \zeta
u\right\Vert _{2}\right)  , \label{reg loc}%
\end{equation}
whenever the right hand side is finite.
\end{theorem}

\begin{proof}
We can repeat the proof of Theorem 3.15 in \cite{BB1} applying
(\ref{first 1/s bound}) to the function $\zeta_{0}u\in\mathcal{H}$, since
$u\in\mathcal{H}_{0}$, and exploiting the identity%
\begin{equation}
\mathcal{L}\left(  \zeta_{0}u\right)  =\left(  \mathcal{L}\zeta_{0}\right)
u+\zeta_{0}\left(  \mathcal{L}u\right)  +2\sum_{i,j=1}^{q}a_{ij}\left(
t\right)  X_{i}\zeta_{0}X_{j}u, \label{L(fg)}%
\end{equation}
and the fact that the operators $\partial_{t}$ and $\widetilde{\Delta}_{h}$
commute, so that $\mathcal{L}$ and $\widetilde{\Delta}_{h}$ still commute.
\end{proof}

Also Proposition 3.16 in \cite{BB1} (Marchaud inequality on Carnot groups)
still holds, with $L^{2}\left(  \mathbb{G}\right)  $ norms replaced with
$L^{2}\left(  \mathbb{G}_{T}\right)  $ norms, and this implies the following
analog of Corollary 3.17 in \cite{BB1}$.$

\begin{corollary}
\label{Coroll Marchaud}Let $u\in$ $\mathcal{H}$, $u\left(  0,\cdot\right)
=0$, and assume that for $\varepsilon\in\left(  0,1\right)  $ and some integer
$m>1$ the seminorm $\left\vert u\right\vert _{m,1+\varepsilon}^{R}$ is finite.
Then%
\[
\left\vert u\right\vert _{1}^{R}\leqslant c\left\{  \left\vert u\right\vert
_{m,1+\varepsilon}^{R}+\left\Vert u\right\Vert _{L^{2}\left(  \mathbb{G}%
_{T}\right)  }\right\}  ,
\]
with $c=c\left(  \mathbb{G}\right)  .$
\end{corollary}

We are now in position to state the first step of our regularity estimate:

\begin{proposition}
\label{Prop Reg step 1}Let $\zeta_{0},\zeta\in C_{0}^{\infty}\left(
\mathbb{G}\right)  $ with $\zeta_{0}\prec\zeta$. There exists $c=c\left(
\zeta_{0},\zeta,\mathbb{G},\nu\right)  >0$ such that
\[
\text{if }u\in\mathcal{H}_{0}\text{ and }\mathcal{L}u\in L^{2}\left(  \left(
0,T\right)  ,W_{X^{R},loc}^{s,2}\left(  \mathbb{G}\right)  \right)  \text{
then }u\in L^{2}\left(  \left(  0,T\right)  ,W_{X^{R},loc}^{1,2}\left(
\mathbb{G}\right)  \right)
\]
and%
\begin{equation}
\left\Vert \zeta_{0}u\right\Vert _{L^{2}\left(  \left(  0,T\right)  ,W_{X^{R}%
}^{1,2}\left(  \mathbb{G}\right)  \right)  }\leqslant c\left(  \left\Vert
\zeta\mathcal{L}u\right\Vert _{L^{2}\left(  \left(  0,T\right)  ,W_{X^{R}%
}^{s,2}\left(  \mathbb{G}\right)  \right)  }+\left\Vert \zeta u\right\Vert
_{L^{2}\left(  \mathbb{G}_{T}\right)  }\right)  .\label{regularity step 1}%
\end{equation}

\end{proposition}

\begin{proof}
Applying to $\zeta_{0}u$ \ Corollary \ref{Coroll Marchaud} and Theorem
\ref{Thm regularity localized} with $m=s+1$ and $\varepsilon=1/s$ we can
write:%
\[
\left\vert \zeta_{0}u\right\vert _{1}^{R}\leqslant c\left\{  \left\vert
\zeta_{0}u\right\vert _{s+1,1+1/s}^{R}+\left\Vert \zeta_{0}u\right\Vert
_{L^{2}\left(  \mathbb{G}_{T}\right)  }\right\}  \leqslant c\left(  \sum
_{j=0}^{s}\left\vert \zeta\mathcal{L}u\right\vert _{j}^{R}+\left\Vert \zeta
u\right\Vert _{L^{2}\left(  \mathbb{G}_{T}\right)  }\right)
\]
From this inequality, by Propositions \ref{Lemma 2 caratt incr fin}
and\ \ref{Thm caratterizzaz increm finiti} we conclude the desired result.
\end{proof}

To iterate this result to higher order derivatives, we first need a
regularization result allowing to apply (\ref{regularity step 1}) to functions
$u$ satisfying weaker assumptions.

\begin{proposition}
\label{Ch4-Prop Local L2}Let $u\in W^{1,2}\left(  \left(  0,T\right)
,L_{loc}^{2}\left(  \mathbb{G}\right)  \right)  ,u\left(  0,\cdot\right)  =0,$
be a weak solution to $\mathcal{L}u=F\in L^{2}\left(  \left(  0,T\right)
,L_{loc}^{2}\left(  \mathbb{G}\right)  \right)  $ in the following sense%
\begin{align}
&  \int_{\mathbb{G}}\left\{  -\partial_{t}u\left(  t,x\right)  \phi\left(
x\right)  +\sum_{i,j=1}^{q}a_{ij}\left(  t\right)  X_{i}X_{j}\phi\left(
x\right)  u\left(  t,x\right)  \right\}  dx\nonumber\\
&  =\int_{\mathbb{G}}F\left(  t,x\right)  \phi\left(  x\right)  dx\text{ for
every }\phi\in C_{0}^{\infty}\left(  \mathbb{G}\right)  \text{ and a.e. }%
t\in\left(  0,T\right)  .\label{def weak sol}%
\end{align}
If $F\in L^{2}\left(  \left(  0,T\right)  ,W_{X^{R},loc}^{s^{2},2}\left(
\mathbb{G}\right)  \right)  $ then $u\in L^{2}\left(  \left(  0,T\right)
,W_{X^{R},loc}^{1,2}\left(  \mathbb{G}\right)  \right)  $ and for every
$\zeta,\zeta_{1}\in C_{0}^{\infty}\left(  \mathbb{G}\right)  $ with
$\zeta\prec\zeta_{1}$ the following estimate holds:%
\begin{equation}
\left\Vert \zeta u\right\Vert _{L^{2}\left(  \left(  0,T\right)  ,W_{X^{R}%
}^{1,2}\left(  \mathbb{G}\right)  \right)  }\leqslant c\left\{  \left\Vert
\zeta_{1}F\right\Vert _{L^{2}\left(  \left(  0,T\right)  ,W_{X^{R}}%
^{s,2}\left(  \mathbb{G}\right)  \right)  }+\left\Vert \zeta_{1}u\right\Vert
_{L^{2}\left(  \mathbb{G}_{T}\right)  }\right\}  \label{Ch4-k=1}%
\end{equation}
with $c=c\left(  \zeta_{0},\zeta,\mathbb{G},\nu\right)  .$
\end{proposition}

\begin{proof}
Let us define the $\varepsilon$-mollified $u_{\varepsilon}$ of $u$ as follows.
For $\phi\in C_{0}^{\infty}\left(  \mathbb{G}\right)  $ such that
\[
\phi\geqslant0,\text{ }\phi\left(  x\right)  =0\text{ for }\left\Vert
x\right\Vert \geqslant1\text{ and }\int_{\mathbb{G}}\phi\left(  x\right)
dx=1,
\]
define, for any $\varepsilon>0,$%
\[
\phi_{\varepsilon}\left(  x\right)  =\varepsilon^{-Q}\phi\left(
D_{\varepsilon^{-1}}x\right)
\]
and%
\[
u_{\varepsilon}\left(  t,x\right)  =\left(  \phi_{\varepsilon}\ast u\right)
\left(  t,x\right)  =\int_{\mathbb{G}}\phi_{\varepsilon}\left(  y\right)
u\left(  t,y^{-1}\circ x\right)  dy=\int_{\mathbb{G}}\phi_{\varepsilon}\left(
x\circ z^{-1}\right)  u\left(  t,z\right)  dz.
\]
Now the function $u_{\varepsilon}$ is smooth with respect to $x$ (as can be
seen computing $X_{I}^{R}u_{\varepsilon}$), while%
\[
\frac{\partial u_{\varepsilon}}{\partial t}=\phi_{\varepsilon}\ast
\frac{\partial u}{\partial t}%
\]
and, for any couple of domains $K\Subset K^{\prime}\Subset\mathbb{G}$ and
$\varepsilon$ small enough,%
\begin{align*}
\left\Vert \frac{\partial u_{\varepsilon}}{\partial t}\left(  t,\cdot\right)
\right\Vert _{L^{2}\left(  K\right)  }  &  \leqslant\left\Vert \frac{\partial
u}{\partial t}\left(  t,\cdot\right)  \right\Vert _{L^{2}\left(  K^{\prime
}\right)  }\\
\left\Vert \frac{\partial u_{\varepsilon}}{\partial t}\right\Vert
_{L^{2}\left(  \left(  0,T\right)  ,L^{2}\left(  K\right)  \right)  }  &
\leqslant\left\Vert \frac{\partial u}{\partial t}\right\Vert _{L^{2}\left(
\left(  0,T\right)  ,L^{2}\left(  K^{\prime}\right)  \right)  }.
\end{align*}
Here we have used Young's inequality in the form%
\begin{equation}
\left\Vert f\ast\phi_{\varepsilon}\right\Vert _{L^{2}\left(  K\right)
}\leqslant\left\Vert f\right\Vert _{L^{2}\left(  K^{\prime}\right)  }
\label{local Young}%
\end{equation}
for $K\Subset K^{\prime}$, and $\varepsilon$ small enough, since
$\phi_{\varepsilon}$ is compactly supported.

Also,%
\[
u_{\varepsilon}\left(  0,x\right)  =\int_{\mathbb{G}}\phi_{\varepsilon}\left(
y\right)  u\left(  0,y^{-1}\circ x\right)  dy=0,
\]
hence $u_{\varepsilon}\in\mathcal{H}_{0}$ and we can apply to $u_{\varepsilon
}$ the estimate proved in Proposition \ref{Prop Reg step 1}:%
\begin{equation}
\left\Vert \zeta u_{\varepsilon}\right\Vert _{L^{2}\left(  \left(  0,T\right)
,W_{X^{R}}^{1,2}\left(  \mathbb{G}\right)  \right)  }\leqslant c\left\{
\left\Vert \zeta_{1}\mathcal{L}\left(  u_{\varepsilon}\right)  \right\Vert
_{L^{2}\left(  \left(  0,T\right)  ,W_{X^{R}}^{s,2}\left(  \mathbb{G}\right)
\right)  }+\left\Vert \zeta_{1}u_{\varepsilon}\right\Vert _{L^{2}\left(
\mathbb{G}_{T}\right)  }\right\}  .\label{Ch4-epsilon local}%
\end{equation}

We claim that%
\begin{equation}
\mathcal{L}\left(  u_{\varepsilon}\right)  =F_{\varepsilon}
\label{Ch4-L f epsi}%
\end{equation}
for $a.e.$ $t$ and $a.e.$ $x$. This is not trivial since $\mathcal{L}u$ just
exists in the above weak sense, hence we cannot simply write $\mathcal{L}%
\left(  u_{\varepsilon}\right)  =\left(  \mathcal{L}u\right)  _{\varepsilon}$.
However, for every $\varphi\in C_{0}^{\infty}\left(  \mathbb{G}\right)  $,
letting%
\[
\mathcal{L}\mathcal{=-\partial}_{t}+\mathcal{A}%
\]
with%
\[
\mathcal{A}u\left(  t,x\right)  =\sum_{i,j=1}^{q}a_{ij}\left(  t\right)
X_{i}X_{j}u\left(  t,x\right)
\]
we can write:%
\[
\int_{\mathbb{G}}\mathcal{L}\left(  u_{\varepsilon}\right)  \left(
t,x\right)  \varphi\left(  x\right)  dx=\int_{\mathbb{G}}-\partial_{t}\left(
u_{\varepsilon}\right)  \left(  t,x\right)  \varphi\left(  x\right)
dx+\int_{\mathbb{G}}u_{\varepsilon}\left(  t,x\right)  \mathcal{A}%
\varphi\left(  x\right)  dx
\]
Next,%
\begin{align*}
&  \int_{\mathbb{G}}\mathcal{A}\varphi\left(  x\right)  \left(  \int%
\phi_{\varepsilon}\left(  y\right)  u\left(  t,y^{-1}\circ x\right)
dy\right)  dx\\
&  =\int_{\mathbb{G}}\phi_{\varepsilon}\left(  y\right)  \left(
\int\mathcal{A}\varphi\left(  x\right)  u\left(  t,y^{-1}\circ x\right)
dx\right)  dy\\
&  =\int_{\mathbb{G}}\phi_{\varepsilon}\left(  y\right)  \left(
\int\mathcal{A}\varphi\left(  y\circ z\right)  u\left(  t,z\right)  dz\right)
dy
\end{align*}
and%
\begin{align*}
\int_{\mathbb{G}}\partial_{t}\left(  u_{\varepsilon}\right)  \left(
t,x\right)  \varphi\left(  x\right)  dx  &  =\int_{\mathbb{G}}\left(  \int%
\phi_{\varepsilon}\left(  y\right)  \partial_{t}u\left(  t,y^{-1}\circ
x\right)  dy\right)  \varphi\left(  x\right)  dx\\
&  =\int_{\mathbb{G}}\phi_{\varepsilon}\left(  y\right)  \left(  \int%
\partial_{t}u\left(  t,y^{-1}\circ x\right)  \varphi\left(  x\right)
dx\right)  dy\\
&  =\int_{\mathbb{G}}\phi_{\varepsilon}\left(  y\right)  \left(  \int%
\partial_{t}u\left(  t,z\right)  \varphi\left(  y\circ z\right)  dz\right)  dy
\end{align*}
letting $\psi_{y}\left(  z\right)  =\varphi\left(  y\circ z\right)  $
\begin{align*}
&  \int_{\mathbb{G}}\mathcal{L}\left(  u_{\varepsilon}\right)  \left(
t,x\right)  \varphi\left(  x\right)  dx\\
&  =\int_{\mathbb{G}}\phi_{\varepsilon}\left(  y\right)  \left(
\int_{\mathbb{G}}-\partial_{t}u\left(  t,z\right)  \psi_{y}\left(  z\right)
+\mathcal{A}\psi_{y}\left(  z\right)  u\left(  z\right)  dz\right)  dy\\
&  =\int_{\mathbb{G}}\phi_{\varepsilon}\left(  y\right)  \left(
\int_{\mathbb{G}}\psi_{y}\left(  z\right)  F\left(  t,z\right)  dz\right)
dy\\
&  =\int_{\mathbb{G}}\phi_{\varepsilon}\left(  y\right)  \left(
\int_{\mathbb{G}}\varphi\left(  x\right)  F\left(  t,y^{-1}\circ x\right)
dx\right)  dy\\
&  =\int_{\mathbb{G}}\varphi\left(  x\right)  \left(  \int_{\mathbb{G}}%
\phi_{\varepsilon}\left(  y\right)  F\left(  t,y^{-1}\circ x\right)
dy\right)  dx=\int_{\mathbb{G}}\varphi\left(  x\right)  F_{\varepsilon}\left(
t,x\right)  dx
\end{align*}
and (\ref{Ch4-L f epsi}) follows. By known properties of the mollifiers, as
$\varepsilon\rightarrow0$ we have $\phi_{\varepsilon}\ast u\rightarrow u$ in
$L^{2}\left(  \mathbb{R}^{N}\right)  $ as soon as $u\in L^{2}\left(
\mathbb{R}^{N}\right)  $. Also, for every left invariant differential operator
$L$ we can write $L\left(  \phi_{\varepsilon}\ast u\right)  =\phi
_{\varepsilon}\ast Lu$ as soon as $Lu$ exists in $L^{2}\left(  \mathbb{R}%
^{N}\right)  .$ Therefore%
\begin{equation}
\zeta_{1}\mathcal{L}\left(  u_{\varepsilon}\right)  =\zeta_{1}F_{\varepsilon
}\rightarrow\zeta_{1}F\text{ in }W_{X}^{k,2}\left(  \mathbb{G}\right)  ,\text{
for }a.e.\text{ }t \label{pointwise}%
\end{equation}
as soon as $F\in L^{2}\left(  \left(  0,T\right)  ,W_{X,loc}^{k,2}\left(
\mathbb{G}\right)  \right)  $.

To prove convergence in $L^{2}\left(  \left(  0,T\right)  ,W_{X^{R},loc}%
^{s,2}\left(  \mathbb{G}\right)  \right)  $ we make the following rough
estimates:%
\begin{align}
\left\Vert \zeta_{1}\mathcal{L}\left(  u_{\varepsilon}\right)  -\zeta
_{1}F\right\Vert _{L^{2}\left(  \left(  0,T\right)  ,W_{X^{R}}^{s,2}\left(
\mathbb{G}\right)  \right)  }  &  \leqslant c\left\Vert \zeta_{1}\left(
\mathcal{L}u\right)  _{\varepsilon}-\zeta_{1}F\right\Vert _{L^{2}\left(
\left(  0,T\right)  ,W^{s,2}\left(  \mathbb{G}\right)  \right)  }\nonumber\\
&  \leqslant c\left\Vert \zeta_{1}F_{\varepsilon}-\zeta_{1}F\right\Vert
_{L^{2}\left(  \left(  0,T\right)  ,W_{X}^{s^{2},2}\left(  \mathbb{G}\right)
\right)  }. \label{normwise}%
\end{align}
In the first inequality we have bounded the Sobolev norm $W_{X^{R}}^{s,2}$ (on
a compact set containing the support of $\zeta_{1}$) with the Euclidean
Sobolev norm on the same domain; in the second one we have exploited
H\"{o}rmander's condition.

We want to show that, for $F\in L^{2}\left(  \left(  0,T\right)
,W_{X,loc}^{s^{2},2}\left(  \mathbb{G}\right)  \right)  $,
\begin{equation}
\left\Vert \zeta_{1}F_{\varepsilon}-\zeta_{1}F\right\Vert _{L^{2}\left(
\left(  0,T\right)  ,W_{X}^{s^{2},2}\left(  \mathbb{G}\right)  \right)
}\rightarrow0.\label{conclusion}%
\end{equation}
Now:%
\begin{align*}
&  \left\Vert \zeta_{1}F_{\varepsilon}-\zeta_{1}F\right\Vert _{L^{2}\left(
\left(  0,T\right)  ,W_{X}^{s^{2},2}\left(  \mathbb{G}\right)  \right)  }%
^{2}\\
&  =\int_{0}^{T}\left\Vert \zeta_{1}F_{\varepsilon}\left(  t,\cdot\right)
-\zeta_{1}F\left(  t,\cdot\right)  \right\Vert _{W_{X}^{s^{2},2}\left(
\mathbb{G}\right)  }^{2}dt\equiv\int_{0}^{T}g_{\varepsilon}\left(  t\right)
dt
\end{align*}
where by (\ref{pointwise}) we already know that%
\[
g_{\varepsilon}\left(  t\right)  \rightarrow0\text{ for }a.e.t\in\left[
0,T\right]  \text{, as }\varepsilon\rightarrow0.
\]
To apply Lebesgue theorem and conclude the desired result we need to bound
$g_{\varepsilon}$ with an integrable function independent of $\varepsilon$.
Now:%
\begin{align*}
\left\Vert \zeta_{1}F_{\varepsilon}\left(  t,\cdot\right)  -\zeta_{1}F\left(
t,\cdot\right)  \right\Vert _{W_{X}^{s^{2},2}\left(  \mathbb{G}\right)  } &
\leqslant\left\Vert \zeta_{1}F_{\varepsilon}\left(  t,\cdot\right)
\right\Vert _{W_{X}^{s^{2},2}\left(  \mathbb{G}\right)  }+\left\Vert \zeta
_{1}F\left(  t,\cdot\right)  \right\Vert _{W_{X}^{s^{2},2}\left(
\mathbb{G}\right)  }\\
\left\Vert \zeta_{1}F_{\varepsilon}\left(  t,\cdot\right)  \right\Vert
_{L^{2}\left(  \mathbb{G}\right)  }^{2} &  \leqslant\left\Vert F_{\varepsilon
}\left(  t,\cdot\right)  \right\Vert _{L^{2}\left(  K\right)  }^{2}%
\leqslant\left\Vert F\left(  t,\cdot\right)  \right\Vert _{L^{2}\left(
K^{\prime}\right)  }^{2}\in L^{1}\left(  0,T\right)
\end{align*}
where $K\Subset K^{\prime}\Subset\mathbb{G}$ and $\varepsilon$ small enough
(see (\ref{local Young})). By (\ref{f*Pg}), we have $X_{i}\left(
F_{\varepsilon}\right)  =\left(  X_{i}F\right)  _{\varepsilon}$, then%
\begin{align*}
X_{i}\left(  \zeta_{1}F_{\varepsilon}\right)   &  =\left(  X_{i}\zeta
_{1}\right)  F_{\varepsilon}+\zeta_{1}\left(  X_{i}F\right)  _{\varepsilon},\\
\left\Vert X_{i}\left(  \zeta_{1}F_{\varepsilon}\left(  t,\cdot\right)
\right)  \right\Vert _{L^{2}\left(  \mathbb{G}\right)  }^{2} &  \leqslant
c\left(  \left\Vert F_{\varepsilon}\left(  t,\cdot\right)  \right\Vert
_{L^{2}\left(  K\right)  }^{2}+\left\Vert \left(  X_{i}F\right)
_{\varepsilon}\left(  t,\cdot\right)  \right\Vert _{L^{2}\left(  K\right)
}^{2}\right)  \\
&  \leqslant c\left(  \left\Vert F\left(  t,\cdot\right)  \right\Vert
_{L^{2}\left(  K^{\prime}\right)  }^{2}+\left\Vert X_{i}F\left(
t,\cdot\right)  \right\Vert _{L^{2}\left(  K^{\prime}\right)  }^{2}\right)
\in L^{1}\left(  0,T\right)  ,
\end{align*}
and an interative reasoning allows to conclude (\ref{conclusion}) Recalling
(\ref{normwise}) and the fact that%
\[
\left\Vert \zeta_{1}u_{\varepsilon}\right\Vert _{L^{2}\left(  \mathbb{G}%
_{T}\right)  }\rightarrow\left\Vert \zeta_{1}u\right\Vert _{L^{2}\left(
\mathbb{G}_{T}\right)  },
\]
we conclude that the right hand side of (\ref{Ch4-epsilon local}) is bounded.
Hence the sequence $\zeta u_{\varepsilon}$ is bounded in $L^{2}\left(  \left(
0,T\right)  ,W_{X^{R}}^{1,2}\left(  \mathbb{G}\right)  \right)  $, and there
exists a subsequence of $\zeta u_{\varepsilon}$ weakly converging in
$L^{2}\left(  \left(  0,T\right)  ,W_{X^{R}}^{1,2}\left(  \mathbb{G}\right)
\right)  $ to some $g$ and in particular weakly converging in $L^{2}\left(
\mathbb{G}_{T}\right)  $ to $\zeta u$. This is enough to say that $\zeta u\in
L^{2}\left(  \left(  0,T\right)  ,W_{X^{R}}^{1,2}\left(  \mathbb{G}\right)
\right)  $. Moreover,%
\begin{align*}
\left\Vert \zeta u\right\Vert _{L^{2}\left(  \left(  0,T\right)  ,W_{X^{R}%
}^{1,2}\left(  \mathbb{G}\right)  \right)  } &  \leqslant\lim\inf\left\Vert
\zeta u_{\varepsilon}\right\Vert _{L^{2}\left(  \left(  0,T\right)  ,W_{X^{R}%
}^{1,2}\left(  \mathbb{G}\right)  \right)  }\\
&  \leqslant c\left\{  \left\Vert \zeta_{1}F\right\Vert _{L^{2}\left(  \left(
0,T\right)  ,W_{X^{R}}^{s,2}\left(  \mathbb{G}\right)  \right)  }+\left\Vert
\zeta_{1}u\right\Vert _{L^{2}\left(  \mathbb{G}_{T}\right)  }\right\}
\end{align*}
hence (\ref{Ch4-k=1}) holds.
\end{proof}

\begin{theorem}
[Regularity estimates in $x$]\label{Thm subsubelliptic}Let $u\in
W^{1,2}\left(  \left(  0,T\right)  ,L_{loc}^{2}\left(  \mathbb{G}\right)
\right)  $, $u\left(  0,\cdot\right)  =0$, be a weak solution to
$\mathcal{L}u=F\in L^{2}\left(  \left(  0,T\right)  ,L_{loc}^{2}\left(
\mathbb{G}\right)  \right)  $ in the sense of (\ref{def weak sol}) and let
$\zeta,\zeta_{1}\in C_{0}^{\infty}\left(  \mathbb{G}\right)  ,\zeta\prec
\zeta_{1}$. Then for any $k=1,2,3,...,$ there exists $c=c\left(  k,\zeta
,\zeta_{1},\mathbb{G},\nu\right)  >0$ such that whenever $\zeta_{1}F\in
L^{2}\left(  \left(  0,T\right)  ,W_{X^{R}}^{k+s^{2}-1,2}\left(
\mathbb{G}\right)  \right)  $ then $\zeta u\in L^{2}\left(  \left(
0,T\right)  ,W_{X^{R}}^{k,2}\left(  \mathbb{G}\right)  \right)  $ and%
\begin{equation}
\left\Vert \zeta u\right\Vert _{W_{X^{R}}^{k,2}\left(  \mathbb{G}_{T}\right)
}\leqslant c\left\{  \left\Vert \zeta_{1}F\right\Vert _{L^{2}\left(  \left(
0,T\right)  ,W_{X^{R}}^{k+s-1,2}\left(  \mathbb{G}\right)  \right)
}+\left\Vert \zeta_{1}u\right\Vert _{L^{2}\left(  \mathbb{G}_{T}\right)
}\right\}  .\label{regularity step k}%
\end{equation}

\end{theorem}

\begin{proof}
We will prove (\ref{regularity step k}) by induction on $k$. For $k=1$ this is
exactly Proposition \ref{Ch4-Prop Local L2}. Assume that
(\ref{regularity step k}) holds up to an integer $k$ and let $u\in
\mathcal{H}_{0}$ such that $\mathcal{L}u\in L^{2}\left(  \left(  0,T\right)
,W_{X^{R}}^{k+s^{2},2}\left(  \mathbb{G}\right)  \right)  $. By the inductive
assumption, $\zeta u\in L^{2}\left(  \left(  0,T\right)  ,W_{X^{R}}%
^{k,2}\left(  \mathbb{G}\right)  \right)  $. Let $X_{I}^{R}$ be a right
invariant differential operator with $\left\vert I\right\vert \leqslant k$,
then $\zeta X_{I}^{R}u\in L^{2}\left(  \mathbb{G}_{T}\right)  $. We would like
to apply Proposition \ref{Ch4-Prop Local L2} to $X_{I}^{R}u$, but in order to
do that we would need to know that $X_{I}^{R}u\in W^{1,2}\left(  \left(
0,T\right)  ,L_{loc}^{2}\left(  \mathbb{G}\right)  \right)  $ with $X_{I}%
^{R}u\left(  0,\cdot\right)  =0$, which is unclear. Then, let $u_{\varepsilon
}$ be the mollified version of $u$ as in the proof of Proposition
\ref{Ch4-Prop Local L2}, so that:%
\[
X_{I}^{R}\left(  u_{\varepsilon}\right)  \left(  t,x\right)  =\int%
_{\mathbb{G}}\left(  X_{I}^{R}\phi_{\varepsilon}\right)  \left(  x\circ
z^{-1}\right)  u\left(  t,z\right)  dz
\]
which is a smooth function in $x$, and since $X_{I}^{R}\phi_{\varepsilon}$ is
integrable (although its $L^{1}\left(  \mathbb{G}\right)  $ norm is not
uniformly bounded with respect to $\varepsilon$) we have
\[
X_{I}^{R}\left(  u_{\varepsilon}\right)  \in L^{2}\left(  \left(  0,T\right)
,L_{loc}^{2}\left(  \mathbb{G}\right)  \right)
\]
(see (\ref{local Young})) and since $\partial_{t}u\in L^{2}\left(  \left(
0,T\right)  ,L_{loc}^{2}\left(  \mathbb{G}\right)  \right)  $, the same is
true for $\partial_{t}X_{I}^{R}\left(  u_{\varepsilon}\right)  $, which equals
$X_{I}^{R}\left(  \partial_{t}u\right)  _{\varepsilon}$. Then%
\[
X_{I}^{R}\left(  u_{\varepsilon}\right)  \in W^{1,2}\left(  \left(
0,T\right)  ,L_{loc}^{2}\left(  \mathbb{G}\right)  \right)
\]
which also implies%
\[
X_{I}^{R}\left(  u_{\varepsilon}\right)  \left(  0,x\right)  =\int%
_{\mathbb{G}}\left(  X_{I}^{R}\phi_{\varepsilon}\right)  \left(  x\circ
z^{-1}\right)  u\left(  0,z\right)  dz=0
\]
since $u\left(  0,\cdot\right)  =0$ in $L^{2}\left(  \mathbb{G}\right)  $. We
claim that%
\[
\mathcal{L}\left(  X_{I}^{R}\left(  u_{\varepsilon}\right)  \right)
=X_{I}^{R}\left(  \mathcal{L}\left(  u_{\varepsilon}\right)  \right)
=X_{I}^{R}\left(  F_{\varepsilon}\right)
\]
at least in weak sense.\ Actually, noting that $\mathcal{L}$ and $X_{I}^{R}$
commute,%
\begin{align*}
\int_{\mathbb{G}}\mathcal{L}\left(  X_{I}^{R}\left(  u_{\varepsilon}\right)
\right)  \left(  t,x\right)  \varphi\left(  x\right)  dx &  =\int_{\mathbb{G}%
}X_{I}^{R}\left(  \mathcal{L}\left(  u_{\varepsilon}\right)  \right)  \left(
t,x\right)  \varphi\left(  x\right)  dx\\
&  =-\int_{\mathbb{G}}\mathcal{L}\left(  u_{\varepsilon}\right)  \left(
t,x\right)  \left(  X_{I}^{R}\varphi\right)  \left(  x\right)  dx
\end{align*}
since $X_{I}^{R}\varphi\in C_{0}^{\infty}\left(  \mathbb{G}\right)  $ and
$\mathcal{L}\left(  u_{\varepsilon}\right)  =F_{\varepsilon}$ for a.e. $t$ and
$x$ (see (\ref{Ch4-L f epsi}))%
\[
=-\int_{\mathbb{G}}F_{\varepsilon}\left(  t,x\right)  \left(  X_{I}^{R}%
\varphi\right)  \left(  x\right)  dx=\int_{\mathbb{G}}X_{I}^{R}\left(
F_{\varepsilon}\right)  \left(  t,x\right)  \varphi\left(  x\right)  dx
\]
for a.e. $t$. Therefore we can apply Proposition \ref{Ch4-Prop Local L2} to
$X_{I}^{R}\left(  u_{\varepsilon}\right)  $ getting%
\begin{align*}
&  \left\Vert \zeta X_{I}^{R}\left(  u_{\varepsilon}\right)  \right\Vert
_{L^{2}\left(  \left(  0,T\right)  ,W_{X^{R}}^{1,2}\left(  \mathbb{G}\right)
\right)  }\\
&  \leqslant c\left\{  \left\Vert \zeta_{1}X_{I}^{R}\left(  F_{\varepsilon
}\right)  \right\Vert _{L^{2}\left(  \left(  0,T\right)  ,W_{X^{R}}%
^{s,2}\left(  \mathbb{G}\right)  \right)  }+\left\Vert \zeta_{1}X_{I}%
^{R}\left(  u_{\varepsilon}\right)  \right\Vert _{L^{2}\left(  \mathbb{G}%
_{T}\right)  }\right\}  .
\end{align*}
Noting that%
\[
\left\Vert \zeta_{1}X_{I}^{R}\left(  u_{\varepsilon}\right)  \right\Vert
_{L^{2}\left(  \mathbb{G}_{T}\right)  }\leqslant\left\Vert \zeta
_{1}X_{I^{\prime}}^{R}\left(  u_{\varepsilon}\right)  \right\Vert
_{L^{2}\left(  \left(  0,T\right)  ,W_{X^{R}}^{1,2}\left(  \mathbb{G}\right)
\right)  }%
\]
for some $I^{\prime}$ with $\left\vert I^{\prime}\right\vert =\left\vert
I\right\vert -1$, we can proceed iteratively getting, for some different
cutoff function $\zeta_{2}\succ\zeta_{1}$,%
\begin{equation}
\left\Vert \zeta X_{I}^{R}\left(  u_{\varepsilon}\right)  \right\Vert
_{L^{2}\left(  \left(  0,T\right)  ,W_{X^{R}}^{1,2}\left(  \mathbb{G}\right)
\right)  }\leqslant c\left\{  \left\Vert \zeta_{2}X_{I}^{R}\left(
F_{\varepsilon}\right)  \right\Vert _{L^{2}\left(  \left(  0,T\right)
,W_{X^{R}}^{s,2}\left(  \mathbb{G}\right)  \right)  }+\left\Vert \zeta
_{2}u_{\varepsilon}\right\Vert _{L^{2}\left(  \mathbb{G}_{T}\right)
}\right\}  .\label{k bound epsi}%
\end{equation}
From this bound, which is uniform with respect to $\varepsilon$, reasoning
like in the proof of Proposition \ref{Ch4-Prop Local L2} we read that, under
the assumption $X_{I}^{R}F\in L^{2}\left(  \left(  0,T\right)  ,W_{X^{R}%
}^{s^{2},2}\left(  \mathbb{G}\right)  \right)  $, which is true as soon as
$F\in L^{2}\left(  \left(  0,T\right)  ,W_{X^{R}}^{k+s^{2},2}\left(
\mathbb{G}\right)  \right)  $, we have the uniform boundedness of%
\[
\left\Vert \zeta X_{I}^{R}\left(  u_{\varepsilon}\right)  \right\Vert
_{L^{2}\left(  \left(  0,T\right)  ,W_{X^{R}}^{1,2}\left(  \mathbb{G}\right)
\right)  },
\]
which implies the weak convergence in $L^{2}\left(  \left(  0,T\right)
,W_{X^{R}}^{1,2}\left(  \mathbb{G}\right)  \right)  $ of (a subsequence of)
$\zeta X_{I}^{R}\left(  u_{\varepsilon}\right)  $ to some $g.$ In particular
the convergence is in $L^{2}\left(  \mathbb{G}_{T}\right)  ,$ which implies
that for every $\eta\in L^{2}\left(  0,T\right)  $ and $\phi\in C_{0}^{\infty
}\left(  \mathbb{G}\right)  $
\[
\int_{0}^{T}\eta\left(  t\right)  \int_{\mathbb{G}}\zeta\left(  x\right)
X_{I}^{R}\left(  u_{\varepsilon}\right)  \left(  t,x\right)  \phi\left(
x\right)  dxdt\rightarrow\int_{0}^{T}\eta\left(  t\right)  \int_{\mathbb{G}%
}g\left(  t,x\right)  \phi\left(  x\right)  dxdt.
\]
Pick the cutoff function $\zeta\left(  x\right)  =1$ on some bounded open set
$\Omega$, then for every $\phi\in C_{0}^{\infty}\left(  \Omega\right)  $ we
have%
\[
\int_{0}^{T}\eta\left(  t\right)  \int_{\mathbb{G}}X_{I}^{R}\left(
u_{\varepsilon}\right)  \left(  t,x\right)  \phi\left(  x\right)
dxdt\rightarrow\int_{0}^{T}\eta\left(  t\right)  \int_{\mathbb{G}}g\left(
t,x\right)  \phi\left(  x\right)  dxdt.
\]
On the other hand,%
\begin{align*}
&  \int_{0}^{T}\eta\left(  t\right)  \int_{\mathbb{G}}X_{I}^{R}\left(
u_{\varepsilon}\right)  \left(  t,x\right)  \phi\left(  x\right)  dxdt\\
&  =\left(  -1\right)  ^{\left\vert I\right\vert }\int_{0}^{T}\eta\left(
t\right)  \int_{\mathbb{G}}u_{\varepsilon}\left(  t,x\right)  X_{I}^{R}%
\phi\left(  x\right)  dxdt\\
&  \rightarrow\left(  -1\right)  ^{\left\vert I\right\vert }\int_{0}^{T}%
\eta\left(  t\right)  \int_{\mathbb{G}}u\left(  t,x\right)  X_{I}^{R}%
\phi\left(  x\right)  dxdt,
\end{align*}
hence%
\[
\int_{0}^{T}\eta\left(  t\right)  \int_{\mathbb{G}}g\left(  t,x\right)
\phi\left(  x\right)  dxdt=\left(  -1\right)  ^{\left\vert I\right\vert }%
\int_{0}^{T}\eta\left(  t\right)  \int_{\mathbb{G}}u\left(  t,x\right)
X_{I}^{R}\phi\left(  x\right)  dxdt
\]
which implies, for a.e. $t$ and a.e. $x\in\Omega$,
\[
g\left(  t,x\right)  =X_{I}^{R}u\left(  t,x\right)
\]
in the sense of weak derivatives. This means that $\zeta X_{I}^{R}u\in
L^{2}\left(  \left(  0,T\right)  ,W_{X^{R}}^{1,2}\left(  \mathbb{G}\right)
\right)  $ and $\zeta X_{I}^{R}\left(  u_{\varepsilon}\right)  \rightarrow
\zeta X_{I}^{R}u$ weakly in $L^{2}\left(  \left(  0,T\right)  ,W_{X^{R}}%
^{1,2}\left(  \mathbb{G}\right)  \right)  $, which also implies, by
(\ref{k bound epsi}),%
\[
\left\Vert \zeta X_{I}^{R}u\right\Vert _{L^{2}\left(  \left(  0,T\right)
,W_{X^{R}}^{1,2}\left(  \mathbb{G}\right)  \right)  }\leqslant c\left\{
\left\Vert \zeta_{2}X_{I}^{R}F\right\Vert _{L^{2}\left(  \left(  0,T\right)
,W_{X^{R}}^{s,2}\left(  \mathbb{G}\right)  \right)  }+\left\Vert \zeta
_{2}u\right\Vert _{L^{2}\left(  \mathbb{G}_{T}\right)  }\right\}  .
\]
So we are done.
\end{proof}

Next, we want to derive from the previous result the fact that, for $F$ smooth
enough, weak solutions to $\mathcal{L}u=F$ are actually strong solutions.
Also, we want to establish H\"{o}lder continuity with respect to time of
solutions (and their space derivatives):

\begin{theorem}
\label{Thm subell plus}Let $u\in W^{1,2}\left(  \left(  0,T\right)
,L_{loc}^{2}\left(  \mathbb{G}\right)  \right)  ,u\left(  0,\cdot\right)  =0$,
be a weak solution to $\mathcal{L}u=F\in L^{2}\left(  \left(  0,T\right)
,L_{loc}^{2}\left(  \mathbb{G}\right)  \right)  $ in the sense of
(\ref{def weak sol}).

(i) For any $k=0,1,2,3,...$
\[
\text{if }F\in L^{2}\left(  \left(  0,T\right)  ,W_{X^{R},loc}^{k+s^{2}%
+2s-1,2}\left(  \mathbb{G}\right)  \right)  \text{ then }u\in W^{1,2}\left(
\left(  0,T\right)  ,W_{X^{R},loc}^{k+2s,2}\left(  \mathbb{G}\right)  \right)
\]
and $u$ is also a strong solution to $\mathcal{L}u=F$. In particular, for
every multiindex $I$ with $\left\vert I\right\vert \leqslant k$ we have
\[
X_{I}^{R}u\in C^{0}\left(  \left[  0,T\right]  ,L_{loc}^{2}\left(
\mathbb{G}\right)  \right)  \text{ and }X_{I}^{R}u\left(  0,\cdot\right)  =0.
\]

(ii) For every (cartesian) derivative $\partial_{x}^{\alpha}$ and $\zeta
,\zeta_{1}\in C_{0}^{\infty}\left(  \mathbb{G}\right)  ,\zeta\prec\zeta_{1}$,
there exists $c=c\left(  \alpha,\zeta,\zeta_{1},\mathbb{G},\nu\right)  >0$ and
a positive integer $h$ such that whenever $F\in L^{2}\left(  \left(
0,T\right)  ,W_{X^{R},loc}^{h,2}\left(  \mathbb{G}\right)  \right)  $ then
\begin{align*}
& \sup_{0<t_{1}<t_{2}<T}\sup_{x\in\mathbb{G}}\frac{\left\vert \zeta\left(
x\right)  \left[  \partial_{x}^{\alpha}u\left(  t_{2},x\right)  -\partial
_{x}^{\alpha}u\left(  t_{1},x\right)  \right]  \right\vert }{\left\vert
t_{2}-t_{1}\right\vert ^{1/2}}\\
& \leqslant c\left\{  \left\Vert \zeta_{1}F\right\Vert _{L^{2}\left(  \left(
0,T\right)  ,W_{X^{R}}^{h,2}\left(  \mathbb{G}\right)  \right)  }+\left\Vert
\zeta_{1}u\right\Vert _{L^{2}\left(  \mathbb{G}_{T}\right)  }\right\}
\end{align*}
and%
\[
\sup_{x\in\mathbb{G}}\left\vert \zeta\left(  x\right)  \partial_{x}^{\alpha
}u\left(  t,x\right)  \right\vert \leqslant c\left\vert t\right\vert
^{1/2}\left\{  \left\Vert \zeta_{1}F\right\Vert _{L^{2}\left(  \left(
0,T\right)  ,W_{X^{R}}^{h,2}\left(  \mathbb{G}\right)  \right)  }+\left\Vert
\zeta_{1}u\right\Vert _{L^{2}\left(  \mathbb{G}_{T}\right)  }\right\}  \text{
}\forall t\in\left[  0,T\right]  .
\]

(iii) In particular, if
\[
\zeta_{1}F\in L^{2}\left(  \left(  0,T\right)  ,C^{\infty}\left(
\mathbb{G}\right)  \right)
\]
then%
\[
\zeta u\in C^{0}\left(  \left[  0,T\right]  ,C^{\infty}\left(  \mathbb{G}%
\right)  \right)  \text{ and }\zeta u_{t}\in L^{2}\left(  \left(  0,T\right)
,C^{\infty}\left(  \mathbb{G}\right)  \right)  .
\]

\end{theorem}

\begin{proof}
Let $\zeta\in C_{0}^{\infty}\left(  \mathbb{G}\right)  $ and $u\in
W_{X^{R},loc}^{2s,2}\left(  \mathbb{G}\right)  $. Inequalities%
\[
\left\Vert \zeta u\right\Vert _{W_{X}^{2,2}\left(  \mathbb{G}\right)
}\leqslant c\left\Vert \zeta u\right\Vert _{W^{2,2}\left(  \mathbb{G}\right)
}\leqslant c\left\Vert \zeta u\right\Vert _{W_{X^{R}}^{2s,2}\left(
\mathbb{G}\right)  }%
\]
show that%
\[
W_{X^{R},loc}^{2s,2}\left(  \mathbb{G}\right)  \subset W_{X,loc}^{2,2}\left(
\mathbb{G}\right)  .
\]

Let $u\in W^{1,2}\left(  \left(  0,T\right)  ,L_{loc}^{2}\left(
\mathbb{G}\right)  \right)  ,u\left(  0,\cdot\right)  =0$, be a weak solution
to $\mathcal{L}u=F\in L^{2}\left(  \left(  0,T\right)  ,W_{X^{R},loc}%
^{h,2}\left(  \mathbb{G}\right)  \right)  $. By Theorem
\ref{Thm subsubelliptic}, if $\zeta_{1}F\in L^{2}\left(  \left(  0,T\right)
,W_{X^{R}}^{k+s^{2}-1,2}\left(  \mathbb{G}\right)  \right)  $, then $\zeta
u\in L^{2}\left(  \left(  0,T\right)  ,W_{X^{R}}^{k,2}\left(  \mathbb{G}%
\right)  \right)  $. In particular, if $h\geq2s+s^{2}-1$ then $u\in
L^{2}\left(  \left(  0,T\right)  ,W_{X,loc}^{2,2}\left(  \mathbb{G}\right)
\right)  $ and this implies that $u$ is actually a strong solution to the
equation $\mathcal{L}u=F$, so that for a.e. $t$ and a.e. $x$ we have%
\begin{equation}
-u_{t}\left(  t,x\right)  +\sum_{i,j=1}^{q}a_{ij}\left(  t\right)  X_{i}%
X_{j}u\left(  t,x\right)  =F\left(  t,x\right)  .\label{identity strong sol}%
\end{equation}
This identity allows to transfer further $x$-regularity of both $F$ and $u$ to
$u_{t}$: if, for some $k=1,2,3,...$, we know that $h\geq k+2s+s^{2}-1$, then
by Theorem \ref{Thm subsubelliptic} $u\in L^{2}\left(  \left(  0,T\right)
,W_{X^{R},loc}^{k+2s,2}\left(  \mathbb{G}\right)  \right)  $, so that
$X_{i}X_{j}u\in L^{2}\left(  \left(  0,T\right)  ,W_{X^{R},loc}^{k,2}\left(
\mathbb{G}\right)  \right)  $, hence by (\ref{identity strong sol}) $u_{t}\in
L^{2}\left(  \left(  0,T\right)  ,W_{X^{R},loc}^{k,2}\left(  \mathbb{G}%
\right)  \right)  $ and $u\in W^{1,2}\left(  \left(  0,T\right)
,W_{X^{R},loc}^{k,2}\left(  \mathbb{G}\right)  \right)  $.

This implies that for $\left\vert I\right\vert \leqslant k$, $X_{I}^{R}u\in
C^{0}\left(  \left[  0,T\right]  ,L_{loc}^{2}\left(  \mathbb{G}\right)
\right)  $. Moreover we can write, for every $t_{1},t_{2}\in\left[
0,T\right]  $ and a.e. $x\in\mathbb{G}$,%
\begin{align}
u\left(  t_{2},x\right)  -u\left(  t_{1},x\right)   &  =\int_{t_{1}}^{t_{2}%
}\partial_{t}u\left(  t,x\right)  dt\label{Holder}\\
X_{I}^{R}u\left(  t_{2},x\right)  -X_{I}^{R}u\left(  t_{1},x\right)   &
=\int_{t_{1}}^{t_{2}}\partial_{t}X_{I}^{R}u\left(  t,x\right)
dt.\label{Holder 2}%
\end{align}
Letting $t_{1}=0$ in (\ref{Holder}) we get%
\[
u\left(  t_{2},x\right)  =\int_{0}^{t_{2}}\partial_{t}u\left(  t,x\right)  dt,
\]
an identity which can also be differentiated with respect to $X_{I}^{R}$,
giving%
\[
X_{I}^{R}u\left(  t_{2},x\right)  =\int_{0}^{t_{2}}X_{I}^{R}\partial
_{t}u\left(  t,x\right)  dt,
\]
which implies%
\[
X_{I}^{R}u\left(  0,\cdot\right)  =0.
\]
This completes the proof of (i). Next, multiplying both sides of
(\ref{Holder 2}) for $\zeta\in C_{0}^{\infty}\left(  \mathbb{G}\right)  $ and
taking $L^{2}\left(  \mathbb{G}\right)  $-norms we get, recalling that
$X_{I}^{R}$ commutes with $\mathcal{L}$:%
\begin{align*}
&  \int_{\mathbb{G}}\zeta\left(  x\right)  ^{2}\left\vert X_{I}^{R}u\left(
t_{2},x\right)  -X_{I}^{R}u\left(  t_{2},x\right)  \right\vert ^{2}dx\\
&  \leqslant\int_{\mathbb{G}}\zeta\left(  x\right)  ^{2}\left\vert \int%
_{t_{1}}^{t_{2}}\left\{  -X_{I}^{R}\mathcal{L}u\left(  t,x\right)
+\sum_{i,j=1}^{q}a_{ij}\left(  t\right)  X_{i}X_{j}X_{I}^{R}u\left(
t,x\right)  \right\}  dt\right\vert ^{2}dx\\
&  \leqslant\int_{\mathbb{G}}\zeta\left(  x\right)  ^{2}\left(  \int_{t_{1}%
}^{t_{2}}\left\{  \left\vert X_{I}^{R}F\left(  t,x\right)  \right\vert
+c_{\nu}\sum_{i,j=1}^{q}\left\vert X_{i}X_{j}X_{I}^{R}u\left(  t,x\right)
\right\vert \right\}  dt\right)  ^{2}dx\\
&  \leqslant\int_{\mathbb{G}}\zeta\left(  x\right)  ^{2}\left\vert t_{2}%
-t_{1}\right\vert \left\{  \int_{0}^{T}\left\vert X_{I}^{R}F\left(
t,x\right)  \right\vert ^{2}dt+c_{\nu}\sum_{i,j=1}^{q}\int_{0}^{T}\left\vert
X_{i}X_{j}X_{I}^{R}u\left(  t,x\right)  \right\vert ^{2}dt\right\}  dx\\
&  =\left\vert t_{2}-t_{1}\right\vert \left\{  \left\Vert \zeta X_{I}%
^{R}F\right\Vert _{L^{2}\left(  \mathbb{G}_{T}\right)  }^{2}+c_{\nu}%
\sum_{i,j=1}^{q}\left\Vert \zeta X_{i}X_{j}X_{I}^{R}u\right\Vert
_{L^{2}\left(  \mathbb{G}_{T}\right)  }^{2}\right\}  .
\end{align*}
By Theorem \ref{Thm subsubelliptic} this implies that%
\begin{align*}
&  \sup_{0<t_{1}<t_{2}<T}\frac{\int_{\mathbb{G}}\zeta\left(  x\right)
^{2}\left\vert X_{I}^{R}u\left(  t_{2},x\right)  -X_{I}^{R}u\left(
t_{2},x\right)  \right\vert ^{2}dx}{\left\vert t_{2}-t_{1}\right\vert }\\
&  \leqslant c\left\{  \left\Vert \zeta_{1}F\right\Vert _{L^{2}\left(  \left(
0,T\right)  ,W_{X^{R}}^{h,2}\left(  \mathbb{G}\right)  \right)  }+\left\Vert
\zeta_{1}u\right\Vert _{L^{2}\left(  \mathbb{G}_{T}\right)  }\right\}  ^{2}%
\end{align*}
for some $h\mathcal{\ }$large enough and any cutoff function $\zeta_{1}$ such
that $\zeta\prec\zeta_{1}$. On the other hand, letting
\[
v\left(  x\right)  =u\left(  t_{2},x\right)  -u\left(  t_{2},x\right)
\]
and noting that every cartesian derivative $\partial_{x}^{\alpha}v\left(
x\right)  $ can be bounded, uniformly on a compact set of $\mathbb{G}$ by a
suitable linear combination of $X_{I}^{R}v$, we arrive to a bound%
\begin{align*}
& \sup_{0<t_{1}<t_{2}<T}\frac{\left\Vert \zeta\left[  \partial_{x}^{\alpha
}u\left(  t_{2},\cdot\right)  -\partial_{x}^{\alpha}u\left(  t_{1}%
,\cdot\right)  \right]  \right\Vert _{L^{2}\left(  \mathbb{G}\right)  }%
}{\left\vert t_{2}-t_{1}\right\vert ^{1/2}}\\
& \leqslant c\left\{  \left\Vert \zeta_{1}F\right\Vert _{L^{2}\left(  \left(
0,T\right)  ,W_{X^{R}}^{h_{1},2}\left(  \mathbb{G}\right)  \right)
}+\left\Vert \zeta_{1}u\right\Vert _{L^{2}\left(  \mathbb{G}_{T}\right)
}\right\}
\end{align*}
for some integer $h_{1}>h$. And since also the sup of $\left\vert \partial
_{x}^{\alpha}u\left(  t_{2},\cdot\right)  -\partial_{x}^{\alpha}u\left(
t_{1},\cdot\right)  \right\vert $ can be bounded, by Sobolev embeddings, by
suitable $L^{2}$ norms of higher order derivatives, we also have a control%
\begin{align*}
& \sup_{0<t_{1}<t_{2}<T}\sup_{x\in\mathbb{G}}\frac{\left\vert \zeta\left(
x\right)  \left[  \partial_{x}^{\alpha}u\left(  t_{2},x\right)  -\partial
_{x}^{\alpha}u\left(  t_{1},x\right)  \right]  \right\vert }{\left\vert
t_{2}-t_{1}\right\vert ^{1/2}}\\
& \leqslant c\left\{  \left\Vert \zeta_{1}F\right\Vert _{L^{2}\left(  \left(
0,T\right)  ,W_{X^{R}}^{h_{2},2}\left(  \mathbb{G}\right)  \right)
}+\left\Vert \zeta_{1}u\right\Vert _{L^{2}\left(  \mathbb{G}_{T}\right)
}\right\}
\end{align*}
for some integer $h_{2}>h_{1}$. Also, since $\partial_{x}^{\alpha}u\left(
0,x\right)  =0,$ this implies%
\[
\sup_{x\in\mathbb{G}}\left\vert \zeta\left(  x\right)  \partial_{x}^{\alpha
}u\left(  t,x\right)  \right\vert \leqslant c\left\vert t\right\vert
^{1/2}\left\{  \left\Vert \zeta_{1}F\right\Vert _{L^{2}\left(  \left(
0,T\right)  ,W_{X^{R}}^{h_{2},2}\left(  \mathbb{G}\right)  \right)
}+\left\Vert \zeta_{1}u\right\Vert _{L^{2}\left(  \mathbb{G}_{T}\right)
}\right\}  ,
\]
This ends the proof of (ii). The previous result also shows that%
\[
\zeta_{1}F\in L^{2}\left(  0,T\right)  ,C^{\infty}\left(  \mathbb{G}\right)
\Longrightarrow\zeta u\in C^{0}\left(  \left[  0,T\right]  ,C^{\infty}\left(
\mathbb{G}\right)  \right)  .
\]
Then the equality%
\[
u_{t}=\sum_{i,j=1}^{q}a_{ij}\left(  t\right)  X_{i}X_{j}u-F
\]
also implies that%
\[
\zeta u_{t}\in L^{2}\left(  \left(  0,T\right)  ,C^{\infty}\left(
\mathbb{G}\right)  \right)  .
\]

\end{proof}

We end this section with an easy example showing that the regularity
properties of the solution cannot be improved for bounded measurable
coefficients $a_{ij}\left(  t\right)  $.

\begin{example}
\label{contres parabolico}Let us consider the uniformly parabolic operator%
\[
\mathcal{L}u=-u_{t}+a\left(  t\right)  u_{xx}%
\]
with $a\in L^{\infty}\left(  \mathbb{R}\right)  $, $a\left(  t\right)  \geq
\nu>0$. The function%
\[
u\left(  t,x\right)  =\exp\left(  -\int_{0}^{t}a\left(  \tau\right)
d\tau\right)  \sin x
\]
satisfies $\mathcal{L}u=0$; $u$ is smooth w.r.t. $x$ and only Lipschitz
continuous w.r.t. $t$. Let%
\[
U\left(  t,x\right)  =t^{\alpha}u\left(  t,x\right)  \text{ for some }%
\alpha\in\left(  \frac{1}{2},1\right)  .
\]
Then $U$ solves the problem%
\[
\left\{
\begin{array}
[c]{l}%
\mathcal{L}U=F\text{ for }x\in\mathbb{R},t>0\\
U\left(  0,x\right)  =0
\end{array}
\right.
\]
with $F\left(  t,x\right)  =-\alpha t^{\alpha-1}u\left(  t,x\right)  $, so
that, as soon as $\alpha>\frac{1}{2},$%
\[
F\in L^{2}\left(  \left(  0,1\right)  \times\mathbb{R}\right)  \text{.}%
\]
Moreover,%
\[
U_{t}\left(  t,x\right)  =\alpha t^{\alpha-1}u\left(  t,x\right)  -t^{\alpha
}a\left(  t\right)  u\left(  t,x\right)  \in L^{2}\left(  \left(  0,T\right)
,C^{\infty}\left(  \mathbb{R}\right)  \right)
\]
Hence%
\[
U\in W^{1,2}\left(  \left(  0,T\right)  ,C^{\infty}\left(  \mathbb{R}\right)
\right)  \cap C^{0,\alpha}\left(  \left[  0,T\right]  ,C^{\infty}\left(
\mathbb{R}\right)  \right)  .
\]
Since $\alpha>\frac{1}{2}$ can be chosen as close to $1/2$ as we want, this
shows that the regularity with respect to $t$ expressed by Theorem
\ref{Thm subell plus} cannot be improved. Also, note that the H\"{o}lder
continuity w.r.t. $t$ cannot be improved to Lipschitz continuity just remaing
far off $t=0$: if we multiply the above $U\left(  t,x\right)  $ for
$\left\vert t-t_{0}\right\vert ^{\alpha}$ we get a similar example exhibiting
a $\alpha$-H\"{o}lder continuity w.r.t. $t$ near $t=t_{0}$.
\end{example}

\section{Regularization of distributional solutions\label{Sec Hypoellipticity}%
}

In this section we want to extend our smoothness result, established in
Theorem \ref{Thm subell plus} (iii) for functions in $W^{1,2}\left(  \left(
0,T\right)  ,L_{loc}^{2}\left(  \mathbb{G}\right)  \right)  $, to more general
distributions. First of all, we need to make precise the distributional
notions that we will use.

\begin{definition}
Let $\Omega\subseteq\mathbb{G}$ be an open set. We will say that $u\in
L^{2}\left(  \left(  0,T\right)  ,\mathcal{D}^{\prime}\left(  \Omega\right)
\right)  $ if $u\in\mathcal{D}^{\prime}\left(  \Omega_{T}\right)  $ and for
every $\phi\in\mathcal{D}\left(  \Omega\right)  $ there exists a function
$h_{\phi}\in L^{2}\left(  0,T\right)  $ such that for every $\psi
\in\mathcal{D}\left(  0,T\right)  ,$%
\[
\left\langle u,\phi\otimes\psi\right\rangle =\int_{0}^{T}h_{\phi}\left(
t\right)  \psi\left(  t\right)  dt.
\]
In this case we will write, more transparently, $h_{\phi}\left(  t\right)
=\left\langle u\left(  t,\cdot\right)  ,\phi\right\rangle $ and%
\[
\left\langle u,\phi\left(  x\right)  \psi\left(  t\right)  \right\rangle
=\int_{0}^{T}\left\langle u\left(  t,\cdot\right)  ,\phi\right\rangle
\psi\left(  t\right)  dt
\]
for every $\phi\in\mathcal{D}\left(  \Omega\right)  \ $and $\psi\in
\mathcal{D}\left(  0,T\right)  $ (and therefore also for every $\psi\in
L^{2}\left(  0,T\right)  $)$.$

Analogously, we will say that $u\in W^{1,2}\left(  \left(  0,T\right)
,\mathcal{D}^{\prime}\left(  \Omega\right)  \right)  $ if $u\in\mathcal{D}%
^{\prime}\left(  \Omega_{T}\right)  $ with both $u$ and its distributional
derivative $\partial_{t}u$ belonging to\ $L^{2}\left(  \left(  0,T\right)
,\mathcal{D}^{\prime}\left(  \Omega\right)  \right)  $.

We will say that $u$ is a distributional solution to $\mathcal{L}u=F$ in
$\Omega_{T}$, with $F\in L^{2}\left(  \left(  0,T\right)  ,\mathcal{D}%
^{\prime}\left(  \Omega\right)  \right)  $ if $u\in W^{1,2}\left(  \left(
0,T\right)  ,\mathcal{D}^{\prime}\left(  \Omega\right)  \right)  $ and:%
\[
\left\langle -\partial_{t}u\left(  t,\cdot\right)  ,\phi\right\rangle
+\sum_{i,j=1}^{q}a_{ij}\left(  t\right)  \left\langle X_{i}X_{j}u\left(
t,\cdot\right)  ,\phi\right\rangle =\left\langle F\left(  t,\cdot\right)
,\phi\right\rangle
\]
for every $\phi\in\mathcal{D}\left(  \Omega\right)  $ and a.e. $t\in\left(
0,T\right)  $, or equivalently:%
\begin{align*}
&  \int_{0}^{T}\left\{  \left\langle -\partial_{t}u\left(  t,\cdot\right)
,\phi\right\rangle +\sum_{i,j=1}^{q}a_{ij}\left(  t\right)  \left\langle
u\left(  t,\cdot\right)  ,X_{i}X_{j}\phi\right\rangle \right\}  \psi\left(
t\right)  dt\\
&  =\int_{0}^{T}\left\langle F\left(  t,\cdot\right)  ,\phi\right\rangle
\psi\left(  t\right)  dt
\end{align*}
$\forall\phi\in\mathcal{D}\left(  \Omega\right)  ,\psi\in L^{2}\left(
0,T\right)  $.
\end{definition}

The proof of a regularity result for distributional solutions usually begins
identifying the given distribution, locally, with some derivative of a
continuous function, in view of the classical result about the local structure
of distributions. For distributions in the class $L^{2}\left(  \left(
0,T\right)  ,\mathcal{D}^{\prime}\left(  \Omega\right)  \right)  $ we could
not find in the literature any reference for a similar result. So we will
explicitly assume that our distribution could be seen, on a fixed domain
compactly contained in $\Omega$, as a space derivative of a suitable function:

\begin{definition}
\label{Def x-finite}Let $u\in L^{2}\left(  \left(  0,T\right)  ,\mathcal{D}%
^{\prime}\left(  \Omega\right)  \right)  $ for some open set $\Omega
\subseteq\mathbb{G}$. We will say that $u$ satisfies the $x$-finite order
assumption on $\Omega$ if:

there exists a function $h\in L^{2}\left(  \left(  0,T\right)  ,L_{loc}%
^{1}\left(  \Omega\right)  \right)  $ and a multiindex $\alpha$ such that%
\begin{equation}
u=\frac{\partial^{\alpha}h}{\partial x^{\alpha}}\text{ in }\mathcal{D}%
^{\prime}\left(  \Omega_{T}\right)  \label{x-finite}%
\end{equation}
that is%
\[
\left\langle u,\phi\left(  x\right)  \psi\left(  t\right)  \right\rangle
=\int_{0}^{T}\left(  \left(  -1\right)  ^{\left\vert \alpha\right\vert }%
\int_{\Omega^{\prime}}h\left(  t,x\right)  \frac{\partial^{\alpha}\phi
}{\partial x^{\alpha}}\left(  x\right)  dx\right)  \psi\left(  t\right)  dt
\]
$\forall\phi\in\mathcal{D}\left(  \Omega\right)  ,\psi\in L^{2}\left(
0,T\right)  $.

If $u\in W^{1,2}\left(  \left(  0,T\right)  ,\mathcal{D}^{\prime}\left(
\Omega\right)  \right)  $, we will say that $u$ satisfies the $x$-finite order
assumption on $\Omega$ if (\ref{x-finite}) holds with $h\in W^{1,2}\left(
\left(  0,T\right)  ,L_{loc}^{1}\left(  \Omega\right)  \right)  $.
\end{definition}

Note that if $u\in W^{1,2}\left(  \left(  0,T\right)  ,\mathcal{D}^{\prime
}\left(  \Omega\right)  \right)  $ satisfies the $x$-finite order assumption
on $\Omega^{\prime}$, then $h\in C^{0}\left(  \left[  0,T\right]  ,L_{loc}%
^{1}\left(  \Omega\right)  \right)  $. In particular, saying that $u\left(
0,\cdot\right)  =0$ means that $h\left(  0,\cdot\right)  =0$ a.e. in $\Omega$.

The aim of this section is to prove that:

\begin{theorem}
\label{Thm hypo}For some bounded domain $\Omega\subset\mathbb{G}$, let $u$ be
a distributional solution to $\mathcal{L}u=F$ in $\Omega_{T}$ with $F\in
L^{2}\left(  \left(  0,T\right)  ,\mathcal{D}^{\prime}\left(  \Omega\right)
\right)  $. Assume that $u$ satisfies the $x$-finite order assumption (see
Definition \ref{Def x-finite}) and $u\left(  0,\cdot\right)  =0$ in $\Omega$.
Then, for every domain $\Omega^{\prime}\Subset\Omega$, if%
\[
F\in L^{2}\left(  \left(  0,T\right)  ,C^{\infty}\left(  \overline{\Omega
}\right)  \right)
\]
then%
\[
u\in C^{0}\left(  \left[  0,T\right]  ,C^{\infty}\left(  \overline
{\Omega^{\prime}}\right)  \right)  \text{ and }u_{t}\in L^{2}\left(  \left(
0,T\right)  ,C^{\infty}\left(  \overline{\Omega^{\prime}}\right)  \right)  .
\]

\end{theorem}

In order to prove Theorem \ref{Thm hypo} we will adapt the technique used in
\cite[\S 4]{BB1} for sublaplacians.

Let us consider the second order differential operator
\[
\mathcal{L}^{R}=\sum_{j=1}^{N}\left(  X_{j}^{R}\right)  ^{2}%
\]
built using the whole canonical base of right invariant vector fields. This is
a right-invariant (but no longer homogeneous) uniformly elliptic operator in
$\mathbb{G}$, which at the origin coincides with the standard Laplacian. The
fundamental solution of the Laplacian can be proved to be a parametrix for
$\mathcal{L}^{R}$:

\begin{proposition}
[{see \cite[Prop. 4.2.]{BB1}}]\label{Lemma-Parametrix}Let $V\subset\mathbb{G}$
be a neighborhood of the origin. There exist $\widetilde{\gamma}\in C^{\infty
}\left(  \mathbb{G}\setminus\left\{  0\right\}  \right)  $ and $\omega\in
C^{\infty}\left(  \mathbb{G}\setminus\left\{  0\right\}  \right)  $, both
supported in $V$, satisfying%
\begin{align}
\left\vert \widetilde{\gamma}\left(  x\right)  \right\vert  &  \leqslant
\frac{c}{\left\vert x\right\vert ^{N-2}}\label{Prop gamma tilde}\\
\left\vert \partial_{x_{i}}\widetilde{\gamma}\left(  x\right)  \right\vert  &
\leqslant\frac{c}{\left\vert x\right\vert ^{N-1}}\text{ \ \ }%
i=1,2,...,N\label{Prop gamma tilde der}\\
\left\vert \omega\left(  x\right)  \right\vert  &  \leqslant\frac
{c}{\left\vert x\right\vert ^{N-2}}\nonumber
\end{align}
and such that in the sense of distributions%
\[
\mathcal{L}^{R}\widetilde{\gamma}=-\delta+\omega.
\]
(Here $\delta$ is the Dirac mass as a distribution in $\mathbb{R}^{N}$).
\end{proposition}

Let us now consider three open sets in $\mathbb{G}$, $\Omega^{\prime}%
\Subset\Omega^{\prime\prime}\Subset\Omega$ and let $V$ be a neighborhood of
the origin such that $V^{-1}\circ\Omega^{\prime}\subset\Omega^{\prime\prime}$.
Define $\widetilde{\gamma}$ as in Proposition \ref{Lemma-Parametrix}, with
$\widetilde{\gamma}$ supported in $V$. The convolution with $\widetilde{\gamma
}$ defines a regularizing operator that acts on functions $u\in L_{loc}%
^{1}\left(  \mathbb{G}_{T}\right)  $ as follows. For every $x\in\Omega
^{\prime}$ and $t\in\left[  0,T\right]  $ we set%
\begin{equation}
T_{V}u\left(  t,x\right)  =\left(  \widetilde{\gamma}\ast u\left(
t,\cdot\right)  \right)  \left(  x\right)  =\int_{\mathbb{G}}\widetilde{\gamma
}\left(  x\circ y^{-1}\right)  u\left(  t,y\right)  dy.\label{T-funct}%
\end{equation}
The subscript $V$ in $T_{V}$ recalls that the definition of the operator
depends on the choice of the neighborhood $V$ used to define
$\widetilde{\gamma}$.

Note that%
\[
T_{V}:L^{2}\left(  \left(  0,T\right)  ,L^{1}\left(  \Omega^{\prime\prime
}\right)  \right)  \longrightarrow L^{2}\left(  \left(  0,T\right)
,L^{1}\left(  \Omega^{\prime}\right)  \right)  .
\]

Namely, for $x\in\Omega^{\prime}$ and $x\circ y^{-1}\in\operatorname{sprt}%
\widetilde{\gamma}$, the point $y=\left(  x\circ y^{-1}\right)  ^{-1}\circ x$
ranges in $V^{-1}\circ\Omega^{\prime}\subset\Omega^{\prime\prime}$, hence
introducing characteristic functions,%
\[
\chi_{\Omega^{\prime}}\left(  x\right)  T_{V}u\left(  t,x\right)
=\int_{\mathbb{G}}\left(  \widetilde{\gamma}\chi_{V}\right)  \left(  x\circ
y^{-1}\right)  u\left(  t,y\right)  \chi_{\Omega^{\prime\prime}}\left(
y\right)  dy,
\]
or%
\begin{equation}
\chi_{\Omega^{\prime}}T_{V}u\left(  t,\cdot\right)  =\widetilde{\gamma}%
\chi_{V}\ast u\left(  t,\cdot\right)  \chi_{\Omega^{\prime\prime}}
\label{sprt T}%
\end{equation}
which by Young's inequality gives, at least for a.e. $t$,%
\[
\left\Vert T_{V}u\left(  t,\cdot\right)  \right\Vert _{L^{1}\left(
\Omega^{\prime}\right)  }\leqslant\left\Vert \widetilde{\gamma}\right\Vert
_{L^{1}\left(  V\right)  }\left\Vert u\left(  t,\cdot\right)  \right\Vert
_{L^{1}\left(  \Omega^{\prime\prime}\right)  }%
\]
and hence%
\[
\left\Vert T_{V}u\right\Vert _{L^{2}\left(  \left(  0,T\right)  ,L^{1}\left(
\Omega^{\prime}\right)  \right)  }\leqslant\left\Vert \widetilde{\gamma
}\right\Vert _{L^{1}\left(  V\right)  }\left\Vert u\right\Vert _{L^{2}\left(
\left(  0,T\right)  ,L^{1}\left(  \Omega^{\prime\prime}\right)  \right)  }.
\]

Also, $T_{V}$ acts on distributions $u\in L^{2}\left(  \left(  0,T\right)
,\mathcal{D}^{\prime}\left(  \Omega\right)  \right)  $ as follows. For every
$\varphi\in\mathcal{D}\left(  \Omega^{\prime}\right)  $ we set%
\begin{equation}
\left\langle T_{V}u\left(  t,\cdot\right)  ,\varphi\right\rangle =\left\langle
u\left(  t,\cdot\right)  ,T_{V}^{\ast}\varphi\right\rangle \label{T-distr}%
\end{equation}
where%
\[
T_{V}^{\ast}\varphi\left(  y\right)  =\int_{\mathbb{G}}\widetilde{\gamma
}\left(  x\circ y^{-1}\right)  \varphi\left(  x\right)  dx.
\]

Observe that the assumption on $V$ implies that $T_{V}^{\ast}\varphi$ is a
test function in $\Omega^{\prime}$. Namely, for $x\in\operatorname{sprt}%
\varphi$ and $x\circ y^{-1}\in\operatorname{sprt}\widetilde{\gamma}$ the point
$y$ ranges in $\Omega^{\prime\prime}\Subset\Omega$. The function $T_{V}^{\ast
}\varphi$ is smooth, as one can see writing
\[
T_{V}^{\ast}\varphi\left(  y\right)  =\int_{\mathbb{G}}\widetilde{\gamma
}\left(  z\right)  \varphi\left(  z\circ y\right)  dx
\]
and computing left invariant derivatives
\[
X_{I}\left(  T_{V}^{\ast}\varphi\right)  \left(  y\right)  =\int_{\mathbb{G}%
}\widetilde{\gamma}\left(  z\right)  \left(  X_{I}\varphi\right)  \left(
z\circ y\right)  dx.
\]
Therefore the pairing (\ref{T-distr}) is well defined. Also, from the previous
identity we easily read that if $\varphi_{k}\rightarrow0$ in $\mathcal{D}%
\left(  \Omega\right)  $ then $T_{V}^{\ast}\varphi_{k}\rightarrow0$ in
$\mathcal{D}\left(  \Omega^{\prime}\right)  $. Hence $T_{V}u\left(
t,\cdot\right)  \in$ $\mathcal{D}^{\prime}\left(  \Omega^{\prime}\right)  $.
Moreover%
\[
\int_{0}^{T}\left\vert \left\langle T_{V}u\left(  t,\cdot\right)
,\varphi\right\rangle \right\vert ^{2}dt=\int_{0}^{T}\left\vert \left\langle
u\left(  t,\cdot\right)  ,T_{V}^{\ast}\varphi\right\rangle \right\vert
^{2}dt<\infty
\]
(just by definition of $L^{2}\left(  \left(  0,T\right)  ,\mathcal{D}^{\prime
}\left(  \Omega\right)  \right)  $), so that
\[
T_{V}:L^{2}\left(  \left(  0,T\right)  ,\mathcal{D}^{\prime}\left(
\Omega\right)  \right)  \longrightarrow L^{2}\left(  \left(  0,T\right)
,\mathcal{D}^{\prime}\left(  \Omega^{\prime}\right)  \right)  .
\]

Next, we need to prove the regularizing properties of $T_{V}$. The following
result is an adaptation of \cite[Prop. 4.4.]{BB1}.

\begin{proposition}
[Regularizing properties of $T_{V}$]\label{Prop_regularizing T}Let
$\Omega^{\prime}\Subset\Omega^{\prime\prime}\Subset\Omega$. There exists a
neighborhood $V$ of the origin such that the operator $T_{V}$ defined in
(\ref{T-distr}) has the following properties.

(1) Let $u\in\mathcal{D}^{\prime}\left(  \left(  0,T\right)  \times
\Omega\right)  $ such that $u=\frac{\partial^{\alpha}}{\partial x^{\alpha}}g$,
for some $g\in L^{2}\left(  \left(  0,T\right)  ,L_{loc}^{1}\left(
\Omega\right)  \right)  $ and multiindex $\alpha.$ Then $T_{V}u\in
L^{2}\left(  \left(  0,T\right)  ,\mathcal{D}^{\prime}\left(  \Omega^{\prime
}\right)  \right)  $ and%
\[
T_{V}u=\sum_{\left\vert \beta\right\vert \leqslant\left\vert \alpha\right\vert
-1}\partial_{x}^{\beta}A_{\beta}\text{ in }\left(  0,T\right)  \times
\Omega^{\prime}%
\]
for suitable $A_{\beta}\in L^{2}\left(  \left(  0,T\right)  ,L_{loc}%
^{1}\left(  \Omega^{\prime}\right)  \right)  $.

(2) Let $u\in L^{2}\left(  \left(  0,T\right)  ,L_{loc}^{p}\left(
\Omega\right)  \right)  $ for some $1\leqslant p<\frac{N}{2}$, then
\[
T_{V}u\in L^{2}\left(  \left(  0,T\right)  ,L^{p^{\prime}}\left(
\Omega^{\prime}\right)  \right)  \text{ for }\frac{1}{p^{\prime}}>\frac{1}%
{p}-\frac{2}{N}%
\]
and%
\[
\left\Vert T_{V}u\right\Vert _{L^{2}\left(  \left(  0,T\right)  ,L^{p^{\prime
}}\left(  \Omega^{\prime}\right)  \right)  }\leqslant c\left\Vert u\right\Vert
_{L^{2}\left(  \left(  0,T\right)  ,L_{loc}^{p}\left(  \Omega\right)  \right)
}.
\]

(3) Let $u\in L^{2}\left(  \left(  0,T\right)  ,L_{loc}^{2}\left(
\Omega\right)  \right)  $ then $T_{V}u\in L^{2}\left(  \left(  0,T\right)
,W_{X}^{1,2}\left(  \Omega^{\prime}\right)  \right)  $.

(4) Let $u\in L^{2}\left(  \left(  0,T\right)  ,C^{\infty}\left(
\overline{\Omega}\right)  \right)  $, then $T_{V}u\in L^{2}\left(  \left(
0,T\right)  ,C^{\infty}\left(  \overline{\Omega^{\prime}}\right)  \right)  .$
\end{proposition}

\begin{remark}
Throughout the next proof, and also in other deductions in the following, all
the stated equalities of the kind $A\left(  t\right)  =B\left(  t\right)  $
hold for a.e. $t$. Rigorously speaking, we should write chains of equalities
of the kind
\[
\int_{0}^{T}A\left(  t\right)  \psi\left(  t\right)  dt=\int_{0}^{T}B\left(
t\right)  \psi\left(  t\right)  dt\text{ for every }\psi\in\mathcal{D}\left(
0,T\right)
\]
and then deduce that $A\left(  t\right)  =B\left(  t\right)  $ a.e.
\end{remark}

\begin{proof}
This proof is an adaptation of the proof of \cite[Prop. 4.4]{BB1}.

(1) Let $u=\ \partial_{x_{i}}\partial_{x}^{\alpha^{\prime}}g$ for some
$\alpha^{\prime}$ with $\left\vert \alpha^{\prime}\right\vert =\left\vert
\alpha\right\vert -1.$ Then, for $\varphi\in C_{0}^{\infty}\left(
\Omega^{\prime}\right)  $%
\begin{align*}
\left\langle T_{V}u\left(  t,\cdot\right)  ,\varphi\right\rangle  &
=\left\langle \partial_{y_{i}}\partial_{y}^{\alpha^{\prime}}g\left(
t,y\right)  ,\int_{\mathbb{G}}\widetilde{\gamma}\left(  x\circ y^{-1}\right)
\varphi\left(  x\right)  dx\right\rangle \\
&  =\left\langle \partial_{y}^{\alpha^{\prime}}g\left(  t,y\right)
,\int_{\mathbb{G}}-\partial_{y_{i}}\left[  \widetilde{\gamma}\left(  x\circ
y^{-1}\right)  \right]  \varphi\left(  x\right)  dx\right\rangle .
\end{align*}
We can write%
\begin{align*}
-\partial_{y_{i}}\left[  \widetilde{\gamma}\left(  x\circ y^{-1}\right)
\right]   &  =-\sum_{k=1}^{N}\left(  \partial_{k}\widetilde{\gamma}\right)
\left(  x\circ y^{-1}\right)  \partial_{y_{i}}\left[  x\circ y^{-1}\right]
_{k}\\
&  =\sum_{k=1}^{N}h_{k}\left(  x\circ y^{-1}\right)  Z_{k}\left(  x,y\right)
\end{align*}
where by (\ref{Prop gamma tilde der}) $h_{k}\left(  z\right)  $ are locally
integrable functions, smooth outside the pole, and $Z_{k}$ are polynomials
(these polynomials also depend on the index $i$ corresponding to
$\partial_{y_{i}}$, but for simplicity we suppress this unimportant index).
Hence%
\begin{align*}
\left\langle T_{V}u\left(  t,\cdot\right)  ,\varphi\right\rangle  &
=\left\langle \partial_{y}^{\alpha^{\prime}}g\left(  t,y\right)
,\int_{\mathbb{G}}\sum_{k=1}^{N}h_{k}\left(  x\circ y^{-1}\right)
Z_{k}\left(  x,y\right)  \varphi\left(  x\right)  dx\right\rangle \\
&  =\left\langle \partial_{y}^{\alpha^{\prime}}g\left(  t,y\right)
,\int_{\mathbb{G}}\sum_{k=1}^{N}h_{k}\left(  w\right)  Z_{k}\left(  w\circ
y,y\right)  \varphi\left(  w\circ y\right)  dw\right\rangle
\end{align*}
since the function $y\mapsto\int_{\mathbb{G}}\sum_{k=1}^{N}h_{k}\left(
w\right)  Z_{k}\left(  w\circ y,y\right)  \varphi\left(  w\circ y\right)  dw$
belongs to $\mathcal{D}\left(  \Omega\right)  $
\[
=\left\langle g\left(  t,y\right)  ,\int_{\mathbb{G}}\sum_{k=1}^{N}%
h_{k}\left(  w\right)  \left(  -1\right)  ^{\left\vert \alpha^{\prime
}\right\vert }\partial_{y}^{\alpha^{\prime}}\left[  Z_{k}\left(  w\circ
y,y\right)  \varphi\left(  w\circ y\right)  \right]  dw\right\rangle
\]
Now,
\[
\left(  -1\right)  ^{\left\vert \alpha^{\prime}\right\vert }\partial
_{y}^{\alpha^{\prime}}\left[  Z_{k}\left(  w\circ y,y\right)  \varphi\left(
w\circ y\right)  \right]  =\sum_{\left\vert \beta\right\vert \leqslant
\left\vert \alpha^{\prime}\right\vert }\left(  D^{\beta}\varphi\right)
\left(  w\circ y\right)  a_{\beta,k}\left(  w,y\right)
\]
for suitable polynomials $a_{\beta,k}$, hence%
\begin{align*}
&  \left\langle T_{V}u\left(  t,\cdot\right)  ,\varphi\right\rangle
=\int_{\mathbb{G}}g\left(  t,y\right)  \left(  \int_{\mathbb{G}}\sum_{k=1}%
^{N}h_{k}\left(  w\right)  \sum_{\left\vert \beta\right\vert \leqslant
\left\vert \alpha^{\prime}\right\vert }\left(  D^{\beta}\varphi\right)
\left(  w\circ y\right)  a_{\beta,k}\left(  w,y\right)  dw\right)  dy\\
&  =\int_{\mathbb{G}}g\left(  t,y\right)  \left(  \int_{\mathbb{G}}\sum
_{k=1}^{N}h_{k}\left(  x\circ y^{-1}\right)  \sum_{\left\vert \beta\right\vert
\leqslant\left\vert \alpha^{\prime}\right\vert }\left(  D^{\beta}%
\varphi\right)  \left(  x\right)  a_{\beta,k}\left(  x\circ y^{-1},y\right)
dx\right)  dy\\
&  =\int_{\mathbb{G}}\sum_{\left\vert \beta\right\vert \leqslant\left\vert
\alpha^{\prime}\right\vert }\left(  D^{\beta}\varphi\right)  \left(  x\right)
\sum_{k=1}^{N}\left(  \int_{\mathbb{G}}g\left(  t,y\right)  h_{k}\left(
x\circ y^{-1}\right)  a_{\beta,k}\left(  x\circ y^{-1},y\right)  dy\right)  dx
\end{align*}
Next, observe that
\[
b_{\beta}\left(  t,x\right)  =\sum_{k=1}^{N}\int_{\mathbb{G}}g\left(
t,y\right)  h_{k}\left(  x\circ y^{-1}\right)  a_{\beta,k}\left(  x\circ
y^{-1},y\right)  dy
\]
belongs to $L^{2}\left(  \left(  0,T\right)  ,L_{loc}^{1}\left(
\Omega^{\prime}\right)  \right)  $, since $g\in L^{2}\left(  \left(
0,T\right)  ,L_{loc}^{1}\left(  \Omega\right)  \right)  $ , $h_{k}\in L^{1}$
and is compactly supported in $V$, $a_{\beta,k}$ are polynomials: for every
$K\Subset\Omega^{\prime}$ there exist $V$ and $K^{\prime}$ such that $K\Subset
K^{\prime}\Subset\Omega^{\prime}$ such that%
\begin{align*}
\int_{K}\left\vert b_{\beta}\left(  t,x\right)  \right\vert dx  &
\leqslant\sum_{k=1}^{N}\int_{K}\int_{\mathbb{G}}\left\vert g\left(
t,y\right)  h_{k}\left(  x\circ y^{-1}\right)  a_{\beta,k}\left(  x\circ
y^{-1},y\right)  \right\vert dydx\\
&  \leqslant c\int_{K^{\prime}}\left\vert g\left(  t,y\right)  \right\vert dy
\end{align*}
so that%
\[
\int_{0}^{T}\left\Vert b_{\beta}\left(  t,\cdot\right)  \right\Vert
_{L^{1}\left(  K\right)  }^{2}dt\leqslant c\int_{0}^{T}\left\Vert g\left(
t,\cdot\right)  \right\Vert _{L^{1}\left(  K^{\prime}\right)  }^{2}dt
\]

Hence, letting%
\[
A_{\beta}\left(  t,x\right)  =\left(  -1\right)  ^{\left\vert \beta\right\vert
}b_{\beta}\left(  t,x\right)
\]
we can write%
\begin{align*}
\left\langle T_{V}u\left(  t,\cdot\right)  ,\varphi\right\rangle  &
=\int_{\mathbb{G}}\sum_{\left\vert \beta\right\vert \leqslant\left\vert
\alpha^{\prime}\right\vert }\left(  -1\right)  ^{\left\vert \beta\right\vert
}A_{\beta}\left(  t,x\right)  \left(  \partial_{x}^{\beta}\varphi\right)
\left(  x\right)  dx\\
&  =\left\langle \sum_{\left\vert \beta\right\vert \leqslant\left\vert
\alpha\right\vert -1}\partial_{x}^{\beta}A_{\beta}\left(  t,\cdot\right)
,\varphi\right\rangle
\end{align*}
with $A_{\beta}\in L^{2}\left(  \left(  0,T\right)  ,L_{loc}^{1}\left(
\Omega^{\prime}\right)  \right)  $, hence $T_{V}u$ has the desired structure.

(2) Inequality%
\[
\left\Vert T_{V}u\left(  t,\cdot\right)  \right\Vert _{L^{p^{\prime}}\left(
\Omega^{\prime}\right)  }\leqslant c\left\Vert u\left(  t,\cdot\right)
\right\Vert _{L^{p}\left(  \Omega\right)  }\text{ for a.e. }t\in\left[
0,T\right]
\]
follows from (\ref{sprt T}) and Young's inequality since, by
(\ref{Prop gamma tilde}), $\widetilde{\gamma}\in L^{r}\left(  \mathbb{G}%
\right)  $ \ for $1\leqslant r<\frac{N}{N-2}$. Taking $L^{2}\left(
0,T\right)  $ norms we get (2).

(3) We know that any derivative $\partial_{x_{i}}\widetilde{\gamma}$
($i=1,2,...,N$) is integrable and supported in $V$, hence each function
$X_{i}^{R}\widetilde{\gamma}$ ($i=1,2,...,N$) is a linear combination with
polynomial coefficients of integrable functions, compactly supported in $V$,
so that $X_{i}^{R}\widetilde{\gamma}\in L^{1}\left(  \mathbb{G}\right)  $.
Also, for a.e. $t\in\left[  0,T\right]  $, $u\chi_{\Omega^{\prime\prime}%
}\left(  t,\cdot\right)  \in L^{2}\left(  \mathbb{G}\right)  $ hence by
Young's inequality%
\[
\chi_{\Omega^{\prime}}T_{V}u\left(  t,\cdot\right)  =\widetilde{\gamma}%
\ast\left(  u\left(  t,\cdot\right)  \chi_{\Omega^{\prime\prime}}\right)  \in
L^{2}\left(  \mathbb{G}\right)  \text{,}%
\]
that is $T_{V}u\left(  t,\cdot\right)  \in L^{2}\left(  \Omega^{\prime
}\right)  $, with%
\[
\left\Vert T_{V}u\left(  t,\cdot\right)  \right\Vert _{L^{2}\left(
\Omega^{\prime}\right)  }\leqslant\left\Vert \widetilde{\gamma}\right\Vert
_{L^{1}}\left\Vert u\left(  t,\cdot\right)  \right\Vert _{L^{2}\left(
\Omega^{\prime\prime}\right)  }%
\]
and
\[
\chi_{\Omega^{\prime}}X_{i}^{R}T_{V}u\left(  t,\cdot\right)  =\left(
X_{i}^{R}\widetilde{\gamma}\right)  \ast\left(  u\chi_{\Omega^{\prime\prime}%
}\right)  \in L^{2}\left(  \mathbb{G}\right)  ,
\]
that is $X_{i}^{R}T_{V}u\in L^{2}\left(  \Omega^{\prime}\right)  $. This holds
for $i=1,2,...,N$ (not just for the first $q$ derivatives). Now, let us recall
that the left invariant vector fields $X_{i}$ ($i=1,2,...,N$) can be written
as linear combinations with polynomial coefficients of the right invariant
vector fields $X_{i}^{R}$. Hence by the boundedness of $\Omega^{\prime}$ we
also have%
\[
X_{i}T_{V}u\left(  t,\cdot\right)  \in L^{2}\left(  \Omega^{\prime}\right)
\text{ for }i=1,2,...,N
\]
with
\[
\left\Vert X_{i}T_{V}u\left(  t,\cdot\right)  \right\Vert _{L^{2}\left(
\Omega^{\prime}\right)  }\leqslant c\sum_{j=1}^{N}\left\Vert X_{j}%
^{R}\widetilde{\gamma}\right\Vert _{L^{1}}\left\Vert u\left(  t,\cdot\right)
\right\Vert _{L^{2}\left(  \Omega^{\prime\prime}\right)  }%
\]
in particular $T_{V}u\in L^{2}\left(  \left(  0,T\right)  ,W_{X}^{1,2}\left(
\Omega^{\prime}\right)  \right)  $ with%
\begin{align*}
\left\Vert T_{V}u\right\Vert _{L^{2}\left(  \left(  0,T\right)  ,W_{X}%
^{1,2}\left(  \Omega^{\prime}\right)  \right)  }  &  \leqslant\left\Vert
\widetilde{\gamma}\right\Vert _{L^{1}}\left\Vert u\right\Vert _{L^{2}\left(
\left(  0,T\right)  \times\Omega^{\prime\prime}\right)  }\\
&  +c\sum_{j=1}^{N}\left\Vert X_{j}^{R}\widetilde{\gamma}\right\Vert _{L^{1}%
}\left\Vert u\right\Vert _{L^{2}\left(  \left(  0,T\right)  \times
\Omega^{\prime\prime}\right)  }.
\end{align*}

(4). Let $u\in C^{\infty}\left(  \overline{\Omega}\right)  $. From
\[
T_{V}u\left(  t,x\right)  =\int_{\mathbb{G}}\widetilde{\gamma}\left(  x\circ
y^{-1}\right)  u\left(  t,y\right)  dy=\int_{\mathbb{G}}\widetilde{\gamma
}\left(  w\right)  u\left(  t,w^{-1}\circ x\right)  dw
\]
we read that for $x\in\Omega^{\prime}$ and any left invariant differential
operator $\mathcal{P}$ we can write%
\[
\mathcal{P}T_{V}u\left(  t,x\right)  =\int_{V}\widetilde{\gamma}\left(
w\right)  \mathcal{P}u\left(  t,w^{-1}\circ x\right)  dw,
\]
showing that $\mathcal{P}T_{V}u\left(  t,\cdot\right)  \in C^{\infty}\left(
\Omega^{\prime}\right)  $. Moreover,%
\[
\max_{x\in\Omega^{\prime}}\left\vert \mathcal{P}T_{V}u\left(  t,x\right)
\right\vert \leqslant c\max_{x\in\Omega}\left\vert \mathcal{P}u\left(
t,x\right)  \right\vert
\]
so that, for every $k=0,1,2,...$%
\[
\left\Vert T_{V}u\left(  t,\cdot\right)  \right\Vert _{C^{k}\left(
\Omega^{\prime}\right)  }\leqslant\left\Vert u\left(  t,\cdot\right)
\right\Vert _{C^{k}\left(  \Omega\right)  }%
\]
and also%
\[
\int_{0}^{T}\left\Vert T_{V}u\left(  t,\cdot\right)  \right\Vert
_{C^{k}\left(  \Omega^{\prime}\right)  }^{2}dt\leqslant\int_{0}^{T}\left\Vert
u\left(  t,\cdot\right)  \right\Vert _{C^{k}\left(  \Omega\right)  }^{2}dt.
\]
Hence $T_{V}u\in L^{2}\left(  \left(  0,T\right)  ,C^{\infty}\left(
\overline{\Omega^{\prime}}\right)  \right)  .$
\end{proof}

\begin{corollary}
\label{Corollary TKf}Let $\Omega^{\prime}\Subset\Omega\Subset\mathbb{G}$. For
every distribution $u\in\mathcal{D}^{\prime}\left(  \left(  0,T\right)
\times\Omega\right)  $ such that $u=\partial_{x}^{\alpha}g$ for some multindex
$\alpha$ and $g\in L^{2}\left(  \left(  0,T\right)  ,L_{loc}^{1}\left(
\Omega\right)  \right)  $ there exist a neighborhood of the origin $V$ and an
integer $K$ such that $\left(  T_{V}\right)  ^{K}u\in L^{2}\left(  \left(
0,T\right)  ,W_{X}^{1,2}\left(  \Omega^{\prime}\right)  \right)  $.
\end{corollary}

The proof follows exactly that of \cite[Corollary 4.5]{BB1}.

\begin{proposition}
\label{ScambioLT}Let $\Omega^{\prime}\Subset\Omega$ and $V$ small enough so
that $V\circ\Omega^{\prime}\Subset\Omega$. Then for any distribution $u\in
L^{2}\left(  \left(  0,T\right)  ,\mathcal{D}^{\prime}\left(  \Omega\right)
\right)  $ and every left invariant operator $\mathcal{P}$ on $\mathbb{G}$ we
have%
\begin{equation}
\mathcal{P}T_{V}u=T_{V}\mathcal{P}u\text{ in }L^{2}\left(  \left(  0,T\right)
,\mathcal{D}^{\prime}\left(  \Omega^{\prime}\right)  \right)  \label{PT=TP}%
\end{equation}
Also, if $u\in W^{1,2}\left(  \left(  0,T\right)  ,\mathcal{D}^{\prime}\left(
\Omega\right)  \right)  $ then%
\begin{equation}
\mathcal{L}T_{V}u=T_{V}\mathcal{L}u\text{ in }L^{2}\left(  \left(  0,T\right)
,\mathcal{D}^{\prime}\left(  \Omega^{\prime}\right)  \right)  \label{LT=TL}%
\end{equation}

\end{proposition}

\begin{remark}
\label{Remark scambio LT}The previous proposition can be obviously iterated
writing%
\begin{align*}
\mathcal{P}T_{V}^{K}u  &  =T_{V}^{K}\mathcal{P}u\text{ in }L^{2}\left(
\left(  0,T\right)  ,\mathcal{D}^{\prime}\left(  \Omega^{\prime}\right)
\right) \\
\mathcal{L}T_{V}^{K}u  &  =T_{V}^{K}\mathcal{L}u\text{ in }L^{2}\left(
\left(  0,T\right)  ,\mathcal{D}^{\prime}\left(  \Omega^{\prime}\right)
\right)
\end{align*}
for any fixed positive integer $K$, provided $V$ is chosen small enough to
have%
\[
\underset{K\text{ times}}{\underbrace{V\circ V\circ...\circ V}}\circ
\Omega^{\prime}\Subset\Omega.
\]

\end{remark}

\begin{proof}
Let $u\in L^{2}\left(  \left(  0,T\right)  ,\mathcal{D}^{\prime}\left(
\Omega\right)  \right)  $, then $T_{V}u\in L^{2}\left(  \left(  0,T\right)
,\mathcal{D}^{\prime}\left(  \Omega^{\prime}\right)  \right)  $ and for every
$\varphi\in\mathcal{D}\left(  \Omega^{\prime}\right)  $ we can write, denoting
by $\mathcal{P}^{\mathcal{\ast}}$ the transpose operator of $\mathcal{P}$ and
recalling that $\mathcal{P}^{\ast}$ is still left invariant,%
\begin{align*}
\left\langle \mathcal{P}T_{V}u\left(  t,\cdot\right)  ,\varphi\right\rangle
&  =\left\langle T_{V}u\left(  t,\cdot\right)  ,\mathcal{P}^{\ast}%
\varphi\right\rangle =\left\langle u\left(  t,y\right)  ,\int_{\mathbb{G}%
}\widetilde{\gamma}\left(  x\circ y^{-1}\right)  \mathcal{P}^{\ast}%
\varphi\left(  x\right)  dx\right\rangle \\
&  =\left\langle u\left(  t,y\right)  ,\int_{\mathbb{G}}\widetilde{\gamma
}\left(  w\right)  \mathcal{P}^{\ast}\varphi\left(  w\circ y\right)
dw\right\rangle \\
&  =\left\langle u\left(  t,y\right)  ,\mathcal{P}^{\ast}\int_{\mathbb{G}%
}\widetilde{\gamma}\left(  w\right)  \varphi\left(  w\circ y\right)
dw\right\rangle \\
&  =\left\langle \mathcal{P}u\left(  t,y\right)  ,\int_{\mathbb{G}%
}\widetilde{\gamma}\left(  w\right)  \varphi\left(  w\circ y\right)
dw\right\rangle \\
&  =\left\langle \mathcal{P}u\left(  t,y\right)  ,\int_{\mathbb{G}%
}\widetilde{\gamma}\left(  x\circ y^{-1}\right)  \varphi\left(  x\right)
dx\right\rangle \\
&  =\left\langle T_{V}\mathcal{P}u\left(  t,\cdot\right)  ,\varphi
\right\rangle .
\end{align*}
where the above equalities holds for a.e. $t$, as usual. This implies
(\ref{PT=TP}).

To prove (\ref{LT=TL}) it is enough to show that
\[
\mathcal{\partial}_{t}T_{V}u=T_{V}\mathcal{\partial}_{t}u\text{ for }u\in
W^{1,2}\left(  \left(  0,T\right)  ,\mathcal{D}^{\prime}\left(  \Omega\right)
\right)  .
\]
Actually, for every $\psi\in\mathcal{D}\left(  0,T\right)  $ and $\varphi
\in\mathcal{D}\left(  \Omega^{\prime}\right)  $ we have%
\begin{align*}
&  \int_{0}^{T}\psi\left(  t\right)  \left\langle \mathcal{\partial}_{t}%
T_{V}u\left(  t,\cdot\right)  ,\varphi\right\rangle dt=\left\langle
\mathcal{\partial}_{t}T_{V}u,\varphi\otimes\psi\right\rangle \\
&  =-\left\langle T_{V}u,\varphi\otimes\mathcal{\partial}_{t}\psi\right\rangle
=-\int_{0}^{T}\mathcal{\partial}_{t}\psi\left(  t\right)  \left\langle
T_{V}u\left(  t,\cdot\right)  ,\varphi\right\rangle dt\\
&  =-\int_{0}^{T}\mathcal{\partial}_{t}\psi\left(  t\right)  \left\langle
u\left(  t,y\right)  ,\int_{\mathbb{G}}\widetilde{\gamma}\left(  x\circ
y^{-1}\right)  \varphi\left(  x\right)  dx\right\rangle dt
\end{align*}%
\begin{align*}
&  =-\left\langle u,T_{V}^{\ast}\varphi\otimes\mathcal{\partial}_{t}%
\psi\right\rangle =\left\langle \mathcal{\partial}_{t}u,T_{V}^{\ast}%
\varphi\otimes\psi\right\rangle \\
&  =\int_{0}^{T}\psi\left(  t\right)  \left\langle \mathcal{\partial}%
_{t}u\left(  t,y\right)  ,\int_{\mathbb{G}}\widetilde{\gamma}\left(  x\circ
y^{-1}\right)  \varphi\left(  x\right)  dx\right\rangle dt\\
&  =\int_{0}^{T}\psi\left(  t\right)  \left\langle T_{V}\mathcal{\partial}%
_{t}u\left(  t,\cdot\right)  ,\varphi\right\rangle dt.
\end{align*}
Hence $\mathcal{\partial}_{t}T_{V}u=T_{V}\mathcal{\partial}_{t}u.$
\end{proof}

\begin{lemma}
\label{tfCinf fCinf}Let $\Omega^{\prime}\Subset\Omega^{\prime\prime}%
\Subset\Omega$ and $u\in L^{2}\left(  \left(  0,T\right)  ,\mathcal{D}%
^{\prime}\left(  \Omega\right)  \right)  $ satisfying the $x$-finite order
assumption in $\Omega$. There exists $V$ small enough so that if%
\[
T_{V}u\in L^{2}\left(  \left(  0,T\right)  ,C^{\infty}\left(  \overline
{\Omega^{\prime\prime}}\right)  \right)
\]
then $u\in L^{2}\left(  \left(  0,T\right)  ,W_{X}^{1,2}\left(  \Omega
^{\prime}\right)  \right)  $.
\end{lemma}

\begin{proof}
For fixed $\Omega^{\prime}\Subset\Omega$ and positive integer $K$ to be chosen
later, there exists a neighborhood $V$ of the origin such that
\[
\underset{K\text{ times}}{\underbrace{V\circ V\circ...\circ V}}\circ
\Omega^{\prime}\Subset\Omega.
\]
Let%
\begin{align*}
\Omega_{j}  &  =\underset{j\text{ times}}{\underbrace{V\circ V\circ...\circ
V}}\circ\Omega^{\prime}\text{ for }j=1,2,...,K\\
\Omega_{0}  &  =\Omega^{\prime}.
\end{align*}
so that $\Omega_{K}\Subset\Omega$. Let $\varphi\in C_{0}^{\infty}\left(
\Omega_{K}\right)  $, using the definition of $T_{V}$ and Lemma
\ref{Lemma-Parametrix} we obtain%
\begin{align*}
\mathcal{L}^{R}T_{V}u\left(  t,\cdot\right)   &  =\mathcal{L}^{R}\left(
\widetilde{\gamma}\ast u\left(  t,\cdot\right)  \right)  =\mathcal{L}%
^{R}\widetilde{\gamma}\ast u\left(  t,\cdot\right) \\
&  =\left(  -\delta+\omega\right)  \ast u\left(  t,\cdot\right)  =-u\left(
t,\cdot\right)  +\omega\ast u\left(  t,\cdot\right)
\end{align*}

We know that $u=\partial_{x}^{\alpha}g$ for some multindex $\alpha$ and $g\in
L^{2}\left(  \left(  0,T\right)  ,L_{loc}^{1}\left(  \Omega\right)  \right)
$. Note that the kernel $\omega$ satisfies the same properties of
$\widetilde{\gamma}$ in terms of support and growth estimate. Then, arguing as
in the proof of Proposition \ref{Prop_regularizing T} we see that%
\[
\omega\ast u\left(  t,\cdot\right)  =\sum_{\left\vert \beta\right\vert
\leqslant A-1}D_{x}^{\beta}A_{\beta}\left(  t,\cdot\right)
\]
with $A_{\beta}\in L^{2}\left(  \left(  0,T\right)  ,L_{loc}^{1}\left(
\Omega\right)  \right)  $ so that
\[
u=\sum_{\left\vert \beta\right\vert \leqslant A-1}D^{\beta}A_{\beta
}-\mathcal{L}^{R}T_{V}u\text{ in }L^{2}\left(  \left(  0,T\right)
,\mathcal{D}^{\prime}\left(  \Omega_{K}\right)  \right)  \text{,}%
\]
with $\mathcal{L}^{R}T_{V}u\in L^{2}\left(  \left(  0,T\right)  ,C^{\infty
}\left(  \overline{\Omega_{K}}\right)  \right)  $ since $T_{V}u\in
L^{2}\left(  \left(  0,T\right)  ,C^{\infty}\left(  \overline{\Omega_{K}%
}\right)  \right)  $ by assumption. Actually, for every $k=0,1,2,...,$%
\[
\left\Vert \mathcal{L}^{R}T_{V}u\left(  t,\cdot\right)  \right\Vert
_{C^{k}\left(  \Omega_{K}\right)  }\leqslant c\left\Vert T_{V}u\left(
t,\cdot\right)  \right\Vert _{C^{k+2}\left(  \Omega_{K}\right)  }%
\]
hence for every $k=0,1,2,...$%
\[
\int_{0}^{T}\left\Vert \mathcal{L}^{R}T_{V}u\left(  t,\cdot\right)
\right\Vert _{C^{k}\left(  \Omega_{K}\right)  }^{2}dt\leqslant c\int_{0}%
^{T}\left\Vert T_{V}u\left(  t,\cdot\right)  \right\Vert _{C^{k+2}\left(
\Omega_{K}\right)  }^{2}dt<\infty.
\]
We can then start again with the identity%
\[
\mathcal{L}^{R}T_{V}u\left(  t,\cdot\right)  =-u\left(  t,\cdot\right)
+\omega\ast u\left(  t,\cdot\right)
\]
where now we know in advance that
\[
u=\sum_{\left\vert \beta\right\vert \leqslant A-1}D^{\beta}A_{\beta}\text{ in
}L^{2}\left(  \left(  0,T\right)  ,\mathcal{D}^{\prime}\left(  \Omega
_{K}\right)  \right)
\]
(the smooth function $\mathcal{L}^{R}T_{V}u$ can be absorbed in this
expression) with $A_{\beta}\in L^{2}\left(  \left(  0,T\right)  ,L_{loc}%
^{1}\left(  \Omega\right)  \right)  $ and, applying iteratively the above
argument, in $k_{1}$ steps we eventually conclude $u\in L^{2}\left(  \left(
0,T\right)  ,L_{loc}^{1}\left(  \Omega_{K-k_{1}}\right)  \right)  $. Hence%
\[
\mathcal{L}^{R}T_{V}u=-u+\omega\ast u\text{ in }L^{2}\left(  \left(
0,T\right)  ,\mathcal{D}^{\prime}\left(  \Omega_{K-k_{1}}\right)  \right)
\]
that is $u$ coincides with $\omega\ast u$ in $L^{2}\left(  \left(  0,T\right)
,\mathcal{D}^{\prime}\left(  \Omega_{K-k_{1}}\right)  \right)  $, modulo the
smooth function $\mathcal{L}^{R}T_{V}u$.

Let us reason again like in the proof of Proposition \ref{Prop_regularizing T}%
: since by Proposition \ref{Lemma-Parametrix}, $\omega\in L^{\frac{N-1}{N-2}%
}\left(  \mathbb{G}\right)  $ we see that $u\in L^{2}\left(  \left(
0,T\right)  ,L^{^{\frac{N-1}{N-2}}}\left(  \Omega_{K-k_{1}-1}\right)  \right)
$; then with $k_{2}$ iterations of this argument we conclude that $u\in
L^{2}\left(  \left(  0,T\right)  ,L^{2}\left(  \Omega_{K-k_{1}-1-k_{2}%
}\right)  \right)  $ and with one more iteration $u\in L^{2}\left(  \left(
0,T\right)  ,W_{X}^{1,2}\left(  \Omega_{K-k_{1}-1-k_{2}-1}\right)  \right)  $.
Picking finally $K=k_{1}+k_{2}+2$ we get the desired assertion.
\end{proof}

\bigskip

\begin{proof}
[Proof of Theorem \ref{Thm hypo}]Fix $\Omega^{\prime}\Subset\Omega
^{\prime\prime}\Subset\Omega^{\prime\prime\prime}\Subset\Omega$. By Corollary
\ref{Corollary TKf}, there exist a positive integer $K$ and a neighborhood $V$
of the origin such that $T_{V}^{K}u\in L^{2}\left(  \left(  0,T\right)
,W_{X}^{1,2}\left(  \Omega^{\prime\prime\prime}\right)  \right)  $. Applying
Corollary \ref{Corollary TKf} also to $\partial_{t}u$, and possibly choosing a
larger integer $K$ and a smaller neighborhood $V$, we can also assume
\[
T_{V}^{K}\partial_{t}u=\partial_{t}T_{V}^{K}u\in L^{2}\left(  \left(
0,T\right)  ,W_{X}^{1,2}\left(  \Omega^{\prime\prime\prime}\right)  \right)
,
\]
so that
\[
T_{V}^{K}u\in W^{1,2}\left(  \left(  0,T\right)  ,W_{X}^{1,2}\left(
\Omega^{\prime\prime\prime}\right)  \right)  .
\]
Let now $U\subseteq V$ a neighborhood of the origin such that%
\[
\underset{2K\text{ times}}{\underbrace{U\circ U\circ...\circ U}}\circ
\Omega^{\prime\prime}\Subset\Omega^{\prime\prime\prime}.
\]
Let%
\begin{align*}
\Omega_{j}  &  =\underset{j\text{ times}}{\underbrace{U\circ U\circ...\circ
U}}\circ\Omega^{\prime\prime}\text{ for }j=1,2,...,2K\\
\Omega_{0}  &  =\Omega^{\prime\prime};
\end{align*}
so that $\Omega_{2K}\Subset\Omega^{\prime\prime\prime}$. Clearly, it is still
true that
\[
T_{U}^{K}u\in L^{2}\left(  \left(  0,T\right)  ,W_{X}^{1,2}\left(
\Omega^{\prime\prime\prime}\right)  \right)
\]
(having replaced the operator $T_{V}$ with $T_{U}$, based on a smaller neighborhood).

By Proposition \ref{ScambioLT} and Remark \ref{Remark scambio LT} we have
\begin{equation}
\mathcal{L}\left(  T_{U}^{K}u\right)  =T_{U}^{K}F\text{ in }L^{2}\left(
\left(  0,T\right)  ,\mathcal{D}^{\prime}\left(  \Omega_{2K}\right)  \right)
.\label{TKLf}%
\end{equation}
Since, $F\in L^{2}\left(  \left(  0,T\right)  ,C^{\infty}\left(
\overline{\Omega}\right)  \right)  $, by point (4) in Proposition
\ref{Prop_regularizing T} we have $T_{U}^{K}F\in L^{2}\left(  \left(
0,T\right)  ,C^{\infty}\left(  \overline{\Omega_{2K}}\right)  \right)  $. By
(\ref{TKLf}) then $\mathcal{L}\left(  T_{U}^{K}u\right)  \in L^{2}\left(
\left(  0,T\right)  ,C^{\infty}\left(  \overline{\Omega_{2K}}\right)  \right)
$ and, since $T_{U}^{K}u\in W^{1,2}\left(  \left(  0,T\right)  ,W_{X}%
^{1,2}\left(  \Omega_{2K}\right)  \right)  $, we can apply Theorem
\ref{Thm subell plus} to conclude that $T_{U}^{K}u\in C^{0}\left(  \left[
0,T\right]  ,C^{\infty}\left(  \overline{\Omega_{2K-1}}\right)  \right)  $
and
\[
T_{U}^{K}u_{t}\in L^{2}\left(  \left(  0,T\right)  ,C^{\infty}\left(
\overline{\Omega_{2K-1}}\right)  \right)  .
\]
Applying Lemma \ref{tfCinf fCinf} to $u$ and $\partial_{t}u$ we see that
$T_{U}^{K-1}u\in W^{1,2}\left(  \left(  0,T\right)  ,W_{X}^{1,2}\left(
\Omega_{2K-2}\right)  \right)  $. Iterating this argument $K$ times shows that
$u\in W^{1,2}\left(  \left(  0,T\right)  ,W_{X}^{1,2}\left(  \Omega
^{\prime\prime}\right)  \right)  $. Since $F\in L^{2}\left(  \left(
0,T\right)  ,C^{\infty}\left(  \overline{\Omega}\right)  \right)  $ we can
apply again Theorem \ref{Thm subell plus} to conclude $u\in C^{0}\left(
\left[  0,T\right]  ,C^{\infty}\left(  \overline{\Omega^{\prime}}\right)
\right)  $ and $u_{t}\in L^{2}\left(  \left(  0,T\right)  ,C^{\infty}\left(
\overline{\Omega^{\prime}}\right)  \right)  $.
\end{proof}

\bigskip

\textsc{Dipartimento di Matematica}

\textsc{Politecnico di Milano}

\textsc{Via Bonardi 9. 20133 Milano, Italy}

\textsc{e-mail: marco.bramanti@polimi.it}

\end{document}